\documentclass{article}
\usepackage{graphicx} % Required for inserting images
\usepackage{amsmath,amssymb,amsthm, epsfig}
\usepackage{color,xcolor}
\usepackage{fullpage}
\usepackage{stackrel}
\usepackage{hyperref}
\usepackage{mathtools}
\usepackage{mathrsfs}
\usepackage{booktabs}
\usepackage{comment}

\newcommand{\eps}{\varepsilon}
\title{\Large\bf Fully Nonlinear Elliptic Grad--Mercier Equations in Weighted Orlicz Spaces}
\author{\it by \smallskip \\ Junior da Silva Bessa, \quad Reshmi Biswas
\,\, and\,\, Mayra Soares}
\newcommand{\intav}[1]{\mathchoice {\mathop{\vrule width 6pt height 3 pt depth  -2.5pt
\kern -8pt \intop}\nolimits_{\kern -6pt#1}} {\mathop{\vrule width
5pt height 3  pt depth -2.6pt \kern -6pt \intop}\nolimits_{#1}}
{\mathop{\vrule width 5pt height 3 pt depth -2.6pt \kern -6pt
\intop}\nolimits_{#1}} {\mathop{\vrule width 5pt height 3 pt depth
-2.6pt \kern -6pt \intop}\nolimits_{#1}}}
\newcommand{\defeq}{\mathrel{\mathop:}=}
\newtheorem{theorem}{Theorem}[section]

\newtheorem{lemma}[theorem]{Lemma}

\newtheorem{proposition}[theorem]{Proposition}

\newtheorem{corollary}[theorem]{Corollary}

\theoremstyle{definition}

\newtheorem{definition}[theorem]{Definition}

\theoremstyle{remark}

\newtheorem{remark}[theorem]{Remark}

\numberwithin{equation}{section}

\date{}

\begin{document}
\maketitle
\begin{abstract}
\noindent In this article, we study the existence and global regularity results for the fully nonlinear elliptic Grad--Mercier type equations with oblique boundary conditions in the context of weighted Orlicz spaces. Our approach employs an asymptotic analysis in which global regularity is transferred from a limit profile, namely, the recession operator associated with the governing operator, using topological and stability methods. In addition to the main regularity result, we derive global weighted Orlicz estimates for the Hessian and establish global Morrey-type estimates for the problem. This article extends the results established by Caffarelli--Tomasetti (Comm. Pure Appl. Math. 76 (3): 604--615, 2023), Zhang et al. (Nonlinearity 39 (2): 025011, 2026), and Bessa (J. Funct. Anal. 286 (4): 110295, 2024).

\medskip
\noindent \textbf{Keywords}: Fully nonlinear equations; Grad-Mercier type equations; Oblique boundary conditions; Weighted Orlicz spaces.
\vspace{0.2cm}
	
\noindent \textbf{AMS Subject Classification: Primary 35A01, 35B65, 35D40, 35J25 Secondary  35A16, 35B30, 35J60}
\end{abstract}
\section{Introduction}

In this manuscript, we investigate the existence and regularity of solutions for a class of  fully nonlinear elliptic equations with oblique boundary conditions in the following setting:
\begin{equation}\label{1.1}
\left\{
\begin{array}{rclcl}
 F(D^2u,Du,u,x) &=& f(x)+\varphi(|\{z\in \Omega\,|\,u(z)\geq u(x)\}|)& \mbox{in} &   \Omega \\
\beta\cdot Du+\gamma u &=& g(x) &\mbox{on}& \partial \Omega,
\end{array}
\right.
\end{equation}
where $\Omega \subset \mathbb{R}^{n}$ ($n \geq 2$) is a bounded domain with regular boundary $\partial \Omega$, the data $f$, $\varphi$, $\beta$, $\gamma$ and $g$ satisfy suitable conditions of regularity. Here, \(F:\text{Sym}(n)\times \mathbb{R}^{n}\times \mathbb{R}\times \Omega\to \mathbb{R}\) is a second-order fully nonlinear elliptic operator, where \(\text{Sym}(n)=\{\mathrm{M}=(m_{ij})_{n\times n}\,|\, \mathrm{M}=\mathrm{M}^{t}\}\). More precisely, there exist positive constants \(\lambda\leq \Lambda\) such that
\[
\lambda \|\mathrm{Y}\|\leq F(\mathrm{X}+\mathrm{Y},\xi,r,x)-F(\mathrm{X},\xi,r,x)\leq \Lambda \|\mathrm{Y}\|,
\]
for all $\mathrm{X}, \mathrm{Y} \in \text{Sym}(n)$ such that $\mathrm{Y} \ge 0$, that is, $\mathrm{Y}$ is a non-negative definite matrix, and for all $(\xi, r, x) \in \mathbb{R}^n \times \mathbb{R} \times \Omega$. Under suitable conditions on the data (see the Subsection \ref{scmr}), we prove the existence of a solution to problem \eqref{1.1} in the \textit{weighted Orlicz–Sobolev} space (see Definition \ref{weightedorliczspaces}).

The study of problem \eqref{1.1} is motivated by the aim of generalizing the Grad–Mercier equation with oblique boundary conditions to fully nonlinear elliptic models. This is a model from plasma physics, where Grad introduced a class of equations to describe the behavior of plasma confined by magnetic fields in a toroidal device (TOKAMAK), see \cite{Grad79}. These equations, are often referred in the literature as queer differential equations (QDEs) or Grad equations, and arise from considering rearrangements of solutions and lead to highly nonlocal formulations. In particular, Grad observed that a simplified version of the plasma equations could be obtained by employing the monotone rearrangement of the solution and  proposed an equation of the form
\[
\Delta u=\varphi(u,u^{*}, (u^{*})',(u^{*})''),
\]
where \(u^{*}\) denotes the increasing rearrangement of \(u\) defined by 
\[
u^{*}(t)=\inf\{s: |\{u<s\}|\geq t\},
\]
and the function \(\varphi\) denotes the plasma dynamics.

Although this formulation provides a tractable reduction of the original physical model, it also introduces some significant mathematical challenges, such as the lack of locality and regularity. Over the years, several authors, for example, Temam \cite{Temam}, Mossino and Temam \cite{MossinoTemam}, and  Laurence and Stredulinsky \cite{LaurenceStredulinsky} proposed approximations and studied special cases, highlighting the intrinsic difficulties of the Grad equations and the need for new analytical tools. \\

Regarding the regularity theory for these models, we  emphasize the seminal work by Caffarelli and Tomasetti \cite{CaffToma21}, where the authors studied the following problem:
\begin{equation}\label{1.2}
\left\{
\begin{array}{rclcl}
 F(D^2u) &=& g_{u}(x)& \mbox{in} &   \Omega \\
u &=& \psi(x) &\mbox{on}& \partial \Omega,
\end{array}
\right.
\end{equation}
where $g_{u}(x) = g(|\{u > u(x)\}|)$, $\Omega \subset \mathbb{R}^{n}$ is a bounded domain with $\partial \Omega \in C^{1,1}$, and $F$ is a convex operator. Taking the assumption that $g$ is continuous and $\psi \in W^{2,p}(\Omega)$ for some $p>n$, the authors proved the existence of a solution to problem \eqref{1.2} and its global $W^{2,p}$ regularity. In addition, they derived the estimate
$$
\|u\|_{W^{2,p}(\Omega)} \leq \mathrm{C} \big(\|u\|_{L^{\infty}(\Omega)} + \|\psi\|_{W^{2,p}(\Omega)} + \|g_{u}\|_{L^{p}(\Omega)}\big),
$$
where $\mathrm{C}>0$ denotes a universal constant. Moreover, the authors established the regularity of the solutions up to $C^{1,\alpha}(\overline{\Omega})$ for every $0<\alpha<1$, providing an improvement of this regularity near the boundary. In particular, when \(F\) is the  Laplacian operator,  $F(\mathrm{M}) = \operatorname{tr}(\mathrm{M})$, and $\psi = 0$, Caffarelli et al. \cite{CafFarRes} showed that the solutions can develop a dead-core phenomenon, corresponding to the region where the solution reaches its maximum. Under suitable assumptions on the reaction term, they also studied uniqueness, sharp regularity, and non-degeneracy estimates for the solutions.

Recently, Zhang et al. \cite{ZJMZ} extended the existence and global $W^{2,p}$ regularity results of Caffarelli and Tomassetti to the model given by
\begin{equation}\label{1.3}
\left\{
\begin{array}{rclcl}
 F(D^2u,Du,u,x) &=& f(x)+g(\mathcal{E}_{u}(x))& \mbox{in} &   \Omega \\
u &=& \psi(x) &\mbox{on}& \partial \Omega,\\
\mathcal{E}_{u}(x) &=& |\{y\in\mathbb{R}^{n}:u(y)\geq u(x)\}| &\mbox{in}& \Omega,
\end{array}
\right.
\end{equation}
where \(g\), \(\psi\)  and \(\Omega\) are in the same setting of Caffarelli-Tomasetti's work \cite{CaffToma21}. In addition, to extend the dependence on the governing operator, the authors weakened the convexity assumption in the matrix variable of \(F\) using an asymptotic approach. More precisely, by using the concept of the recession operator defined by
\begin{equation}\label{recessionop}
F^{\star}(\mathrm{M},\zeta,r,x)=\lim_{\mu\to 0^{+}}\mu F\left(\frac{1}{\mu}\mathrm{M},\zeta,r,x\right), 
\end{equation}
for all $(\mathrm{M},\zeta,r,x)\in \operatorname{Sym}(n)\times\mathbb{R}^{n}\times \mathbb{R}\times \Omega$, this terminology for the class of operators originates from the work of Giga and Sato \cite{GS01} in the study of Hamilton–Jacobi PDEs. Under the convexity assumption on $F^{\star}$ with respect to the matrix variable, the authors established the existence of solutions to the problem \eqref{1.3} in $W^{2,p}(\Omega)$ for $n<p<+\infty$, together with the following estimate
\[
\|u\|_{W^{2,p}(\Omega)} \leq \mathrm{C} \big(\|u\|_{L^{\infty}(\Omega)} +\|\psi\|_{W^{2,p}(\Omega)} + \|g_{u}\|_{L^{p}(\Omega)}+\|f\|_{L^{p}(\Omega)}\big).
\]
In addition, the authors proved global $p$-BMO estimates for the case $g=0$.
\medskip

Concerning the problems of type \eqref{1.1}, substantial advances have been made in the regularity theory when \(\varphi=0\).  In parallel with the aforementioned references, we emphasize that Byun and Han in \cite{BH20} established \(W^{2,p}\) estimates for viscosity solutions of convex fully nonlinear elliptic problem:
\begin{equation*}
\left\{
\begin{array}{rclcl}
 F(D^2u,Du,u,x) &=& f(x)& \mbox{in} &   \Omega \\
\beta\cdot Du&=& 0 &\mbox{on}& \partial \Omega,
\end{array}
\right.
\end{equation*}
with the following estimate
\[
\|u\|_{W^{2,p}(\Omega)}\leq\mathrm{C}(\|u\|_{L^{\infty}(\Omega)}+\|f\|_{L^{p}(\Omega)}).
\]

 In the absence of the convexity assumption and within an asymptotic framework, Bessa et al. in  \cite{Bessa} extended these \(W^{2,p}\) estimates for the problems of type \eqref{1.1} when \(\varphi = 0\) (see also \cite{ZZZ21} when \(\gamma=g=0\)). More precisely, assuming a priori \(C^{1,1}\) estimates for the oblique boundary value problem governed by \(F^{\star}\), the authors proved that, if \(\beta,\gamma,g\in C^{1,\alpha}(\partial \Omega)\), the viscosity solutions of the problem
 \begin{equation}\label{obliqprob}
\left\{
\begin{array}{rclcl}
 F(D^2u,Du,u,x) &=& f(x)& \mbox{in} &   \Omega \\
\beta\cdot Du+\gamma u&=& g(x) &\mbox{on}& \partial \Omega,
\end{array}
\right.
\end{equation}
enjoy global \(W^{2,p}\) regularity, satisfying the following estimate:
\[
\|u\|_{W^{2,p}(\Omega)}\leq\mathrm{C}(\|u\|_{L^{\infty}(\Omega)}+\|f\|_{L^{p}(\Omega)}+\|g\|_{C^{1,\alpha}(\partial \Omega)}).
\]
Furthermore, the authors obtained several applications of these estimates, namely: global estimates for the obstacle problem with oblique boundary conditions, p-BMO type estimates, and density results for solutions. Along these lines, results have been obtained in more general weighted function spaces, such as Lorentz spaces (see \cite{BR}, \cite{ZZ22}). In particular, we refer to \cite{Bessa24} to highlight such results in weighted Orlicz spaces.

Next, in Table 1, we summarize these developments with respect to the problem \eqref{obliqprob}. In addition, notable contributions concerning regularity results for oblique boundary value problems for fully nonlinear elliptic equations can be found, for instance, in the non-exhaustive list \cite{BRS,BJ1,BKO,LiZhang,Lieb01}.

\begin{table}[h!]
\centering
\begin{tabular}{|c|c|c|c|}
\hline
\textbf{Governing operator} & \textbf{Boundary data} & \textbf{Source term space} & \textbf{Reference} \\
\hline
Convex & $\beta\in C^{1,\alpha}$ and $\gamma=g=0$ & \(L^{p}\) space & \cite[Theorem 4.6]{BH20} \\ \hline
Asymptotic convex  & $\beta\in C^{1,\alpha}$ and $\gamma=g=0$ & $L^{p}$ space & \cite[Theorem 2.3]{ZZZ21} \\ \hline
Relaxed convexity  & $\beta,\gamma,g\in C^{1,\alpha} $ & $L^{p}$ space & \cite[Theorem 1.2]{Bessa} \\ \hline
Convex & $\beta\in C^{1,\alpha}$ and $\gamma=g=0$ & Weighted Lorentz space & \cite[Theorem 1.1]{ZZ22} \\ \hline
Relaxed convexity & $\beta,\gamma,g\in C^{1,\alpha} $ & Weighted Lorentz space & \cite[Theorem 1.7]{BR} \\ \hline
Relaxed convexity & $\beta,\gamma,g\in C^{1,\alpha} $ & Weighted Orlicz space & \cite[Theorem 1.6]{Bessa24} \\ \hline
%\bottomrule
\end{tabular}
\caption{Summary of Regularity Results for Oblique Problems in Different Function Spaces}
\label{tab:regularidade}
\end{table}

Despite the advances mentioned above, the study of the existence and regularity of the problem \eqref{1.1} remains open in the literature,  mainly due to two difficulties. First, the equation in \eqref{1.1} exhibits a source term of semilinear type within the domain, which prevents a direct application of the regularity results available for fully nonlinear models such as those highlighted above. Second, the presence of the oblique boundary operator
\begin{equation}\label{opolbi}
\mathcal{G}(\vec{\xi},s,x):=\beta(x)\cdot \vec{\xi}+\gamma(x)s, 
\qquad (\vec{\xi},s,x)\in\mathbb{R}^{n}\times\mathbb{R}\times \partial \Omega,
\end{equation}
introduces an additional layer of difficulty, since on $\partial\Omega$ the boundary condition constitutes a first-order partial differential equation. In particular, the analysis of \eqref{1.1} cannot be carried out by treating the interior equation and the boundary condition separately.
The main goal of this paper,  precisely, is to address this gap by establishing the existence and global regularity results for the problem \eqref{1.1} within the framework of weighted Orlicz spaces.
\subsection{Hypotheses and  Main Results}\label{scmr}

Throughout this manuscript, we assume the following structural conditions for problem \eqref{1.1}:

\begin{enumerate}
\item[\bf$(A_1)$](\textbf{Operator}) Let $F:\text{Sym}(n)\times \mathbb{R}^n \times \mathbb{R}\times \Omega\to\mathbb{R}$ be a continuous function. There are constants $0 < \lambda \le \Lambda$, $a\geq 0$ and $b\geq 0$ such that
\begin{eqnarray*}
\mathcal{P}^{-}_{\lambda,\Lambda}(\mathrm{X}-\mathrm{Y}) - a |\zeta_1-\zeta_2| -b|r_1-r_2| &\le& F(\mathrm{X}, \zeta_1,r_1,x)-F(\mathrm{Y},\zeta_2,r_2,x) \nonumber \\
&\le& \mathcal{P}^{+}_{\lambda, \Lambda}(\mathrm{X}-\mathrm{Y})+ a |\zeta_1-\zeta_2| + b|r_1-r_2| \label{5}
\end{eqnarray*}
for all $\mathrm{X},\mathrm{Y} \in \text{Sym}(n)$, $\zeta_1,\zeta_2 \in \mathbb{R}^n$, $r_1,r_2 \in \mathbb{R}$, $x \in \Omega$, where \(\mathcal{P}^{\pm}_{\lambda,\Lambda}\) denote the \textit{Pucci's extremal operators} defined by
\begin{equation*}
\mathcal{P}^{+}_{\lambda,\Lambda}(\mathrm{X}) \defeq  \Lambda \sum_{e_i >0} e_i +\lambda \sum_{e_i <0} e_i \quad \text{and} \quad \mathcal{P}^{-}_{\lambda,\Lambda}(\mathrm{X}) \defeq \Lambda \sum_{e_i <0} e_i + \lambda \sum_{e_i >  0} e_i,
 \end{equation*}
 where $e_i = e_i(\mathrm{X})$ for $1\leq i\leq n$ denote the eigenvalues of $\mathrm{X}$. By the normalization condition, $F(0,0,0,x)=0$, for all $x\in \Omega$. 
  An operator fulfilling such a uniform ellipticity will be called \emph{$(\lambda, \Lambda, a, b)$ - elliptic operator. }
{\item[\bf$(A_2)$] (\textbf{Hypothesis on the data}) Let $|f|^{n} \in L^{\Phi}_{\omega}(\Omega)$ for some $\omega\in \mathfrak{A}_{i(\Phi)}$, where $\mathfrak{A}_{i(\Phi)}$ denotes the  \textit{Muckenhoupt class} of the \textit{lower index} of $\Phi\in\Delta_{2}\cap \nabla_{2}$ (see Subsection \ref{wos}). Assume the function $\varphi:[0,|\Omega|]\to \mathbb{R}$ is  continuous and the vector field $\beta:\partial \Omega\to \mathbb{R}^{n}$ belongs to  $C^{2}( \partial \Omega; \mathbb{R}^n)$ for some $\alpha\in(0,1)$ and satisfies, 
$$
\beta(x) \cdot \mathbf{\vec{n}}(x) \ge \delta_0 \quad \textrm{on} \quad \partial \Omega \quad \textrm{and} \quad \|\beta\|_{L^{\infty}(\partial \Omega)} \le 1,
$$
for some positive constant $\delta_0$ and $\mathbf{\vec{n}}$ denotes the inner inward  normal vector to $\Omega$. Moreover, $ \gamma \in C^{2}(\partial \Omega)$ with $\gamma \le 0$ and $g\in C^{1,\alpha}(\partial \Omega)$.}
%%%%%%%%%%%%%%%%%%
%%%%%%%%%%%%%%%
\item[\bf($A_3$)](\textbf{$C^{1,1}$ - interior estimates})
The recession operator $F^{\star}$, associated with the operator $F$, is well defined and satisfies $C^{1,1}_{\mathrm{loc}}$ \textit{a priori}  estimates. That is, if $F^{\star}(D^{2}h) = 0$ in $\mathrm{B}_1$, in the viscosity sense, then $h \in C^{1,1}(\overline{\mathrm{B}_{1/2}})$, and
$$
\|h\|_{C^{1,1}(\overline{\mathrm{B}_{1/2}})} \leq \mathfrak{C}_1 \|h\|_{L^{\infty}(\mathrm{B}_1)},
$$
for some constant $\mathfrak{C}_1 \geq 1$.

\item[\bf($A_4$)] (\textbf{$C^{1,1}$ - boundary estimates}) 
The recession operator $F^{\star}$, associated with the operator $F$, is well defined and  satisfies up-to-the-boundary $C^{1,1}$ \textit{a priori} estimates. That is, for any $x_0 \in \mathrm{B}^{+}_{1}$ and $g_0 \in C^{1,\alpha}(\overline{\mathrm{T}_1})$, for some $\alpha \in (0,1)$, there exists a solution $h \in C^{1,1}(\mathrm{B}^+_1) \cap C^0(\overline{\mathrm{B}^+_1})$ of the boundary value problem
$$
\left\{
\begin{array}{rclcl}
 F^{\star}(D^2 h,x_0) &=& 0& \mbox{in} &   \mathrm{B}^+_1 \\
\mathcal{G}(Dh,h,x)&=& g_{0}(x) &\mbox{on}& \mathrm{T}_1
\end{array}
\right.
$$
satisfying $
\|h\|_{C^{1,1}\left(\overline{\mathrm{B}^{+}_{1/2}}\right)} \le \mathfrak{C}_2 \left( \|h\|_{L^{\infty}(\mathrm{B}^{+}_1)} + \|g_0\|_{C^{1,\alpha}(\overline{\mathrm{T}_1})} \right),
$
for some constant $\mathfrak{C}_2 > 0$.
\item[\bf$(A_5)$](\textbf{Oscillation of the coefficients}) Given $x_0\in \Omega$, define 
$$
\Psi_{F}(x; x_0) \defeq \sup_{\mathrm{X} \in \textrm{Sym}(n)} \frac{|F(\mathrm{X},0,0,x) - F(\mathrm{X},0,0,x_0)|}{1+\|\mathrm{X}\|},
$$
which measures the oscillation of the coefficients of the operator $F$ around $x_0$. Denote for short $\Psi_{F}(x; 0) = \Psi_F(x)$. In this context, $\Psi_{F^{\star}}$ is a Hölder continuous function, in the $L^p$-average sense, for every $\mathrm{X} \in \mathrm{Sym}(n)$.  That is, there exist universal constants, depending only on $n$, $\lambda$, $\Lambda$, $p$, $\delta_0$, $\|\gamma\|_{C^{1,\alpha}(\partial \Omega)}$, and $\|\beta\|_{C^{1,\alpha}(\partial \Omega)}$, with $\alpha^{\prime} \in (0,1)$, $\bar{\theta} > 0$, and $0 < r_0 \le 1$, such that
$$
\left( \intav{\mathrm{B}_r(x_0) \cap \Omega} \Psi_{F^{\star}}(x,x_0)^{p} dx \right)^{1/p} \le \bar{\theta} r^{\alpha^{\prime}}, \quad 
\text{ for } x_0 \in \overline{\Omega} \mbox{ and } 0 < r \le r_0.$$ 

\end{enumerate}

In order to guarantee existence, uniqueness, and a comparison principle for the oblique problem, the following additional conditions are needed: 
\begin{enumerate}
\item[\bf $(E_1)$] There exists a modulus of continuity $\eta: [0,+\infty) \to [0,+\infty)$ with $\eta(0)=0$, such that
	$$
	F(\mathrm{X}_1, \zeta, r,x_1) - F(\mathrm{X}_2,\zeta,r,x_2) \le \eta\left(|x_1-x_2|\right)\left[(|q| +1) + \alpha_0 |x_1-x_2|^2\right]
	$$
	for any $x_1,x_2 \in \Omega$, $\zeta \in \mathbb{R}^n$, $r \in \mathbb{R}$, $\alpha_0 >0$ and $\mathrm{X}_1,\mathrm{X}_2 \in \textrm{Sym}(n)$ satisfying
	$$
	- 3 \alpha_0
	\begin{pmatrix}
		\mathrm{Id}_n& 0 \\
		0& \mathrm{Id}_n
	\end{pmatrix}
	\leq
	\begin{pmatrix}
		\mathrm{X}_2&0\\
		0&-\mathrm{X}_1
	\end{pmatrix}
	\leq
	3 \alpha_0
	\begin{pmatrix}
		\mathrm{Id}_n & -\mathrm{Id}_n \\
		-\mathrm{Id}_n& \mathrm{Id}_n
	\end{pmatrix},	
	$$
	where $\mathrm{Id}_n$ is the identity matrix.
\item[\bf$(E_2)$] $F$ is a \textit{proper} operator. Specifically, there exists a positive constant \(d\) such that,
$$
F(\mathrm{X},\xi,s,x)-F(\mathrm{X},\xi,r,x)\geq d\cdot (r-s),
$$
for any $\mathrm{X} \in \text{Sym}(n)$, $r,s \in \mathbb{R}$, with $s\leq r$, $x \in \Omega$, $\xi \in \mathbb{R}^n$.
\end{enumerate}

We are now able to state our first main result regarding existence and a $
W^{2,\Upsilon}_{\omega}
$ \emph{a priori} estimate.

\begin{theorem}\label{T1}
Assume the structural conditions {\bf $(A_1)$}-{\bf $(A_5)$} and {\bf $(E_1)$}-{\bf$(E_2)$} hold true and that \(\partial\Omega\in C^{3}\). Then, there exists  a \(L^{p}\)-viscosity solution \(u\) of \eqref{1.1},  where \(p=\bar{p}n\) for the constant \(\bar{p}>1\) given in Lemma \ref{mergulhoorliczlebesgue}. Furthermore, \(u\in W^{2,\Upsilon}_{\omega}(\Omega)\) for \(\Upsilon(t)=\Phi(t^{n})\), and the following estimate holds,
\begin{eqnarray*}
\|u\|_{W^{2,\Upsilon}_{\omega}(\Omega)}\leq \mathrm{C}(\|f\|_{L^{\Upsilon}_{\omega}(\Omega)}+\|\varphi_{u}\|_{L^{\Upsilon}_{\omega}(\Omega)}+\|g\|_{C^{1,\alpha}(\partial \Omega)}),
\end{eqnarray*}
where \(\mathrm{C}>0\) is a constant depending only on \(n\), \(\lambda\), \(\Lambda\), {\(a\), \(b\),} \(\bar{p}\), {\(p_{2}\)}, \(\Phi\), \(\omega\), \(\delta_0\), \(\bar{\theta}\), \(\alpha^{\prime}\), \(\|\beta\|_{C^{2}(\partial \Omega)}\), \(\|\gamma\|_{C^{2}(\partial \Omega)}\) and \(\|\partial \Omega\|_{C^{2,\alpha}}\).
\end{theorem}

Provided that the classical estimate in
$
W^{2,\Upsilon}_{\omega}
$
controls only the integrability of the Hessian, there exists a natural motivation for pursuing a stronger estimate on the
$
L^{\Upsilon}_{\omega}\text{-}\mathrm{BMO}
$
framework, which also captures the local oscillatory behavior of the Hessian. 
More precisely, estimates in the generalized BMO-type spaces are borderline regularity results, lying between pure integrability estimates and H\"older continuity. They are particularly relevant in situations where the data fail to possess sufficient smoothness to yield \(C^{2,\alpha}\)-regularity, but still exhibit controlled mean oscillation. Therefore, obtaining a
$
L^{\Upsilon}_{\omega}\text{-}\mathrm{BMO}
$
estimate can be viewed as a refinement of the
$
W^{2,\Upsilon}_{\omega}
$-regularity
theory, providing sharper control on the second derivatives and extending the regularity theory to critical oscillation spaces.

In this sense, our second main  result presents a
$L^{\Upsilon}_{\omega}\text{-}\mathrm{BMO}
$
estimate for the Hessian, which  provides a finer quantitative description of the regularity of \(D^{2}u\).
For this purpose, we  need to assume the following extra hypothesis:
\begin{enumerate}
\item[\bf ($A_4$)$^{\star}$] (\textbf{$C^{2,\rho}$ - boundary estimates}) The recession operator $F^{\star}$, associated with $F$,  is well defined and fulfills a $C^{2, \rho}$ \textit{a priori} estimate up to the boundary, for some $\rho \in (0,1)$, i.e., for any $g_0 \in C^{1, \rho}(\mathrm{T}_1)$, $\beta,\gamma \in C^{1, \rho}(\mathrm{T}_1)$, there exists $\mathrm{c}^{\ast}>0$, depending only on universal parameters, such that any viscosity solution of
$$
\left\{
\begin{array}{rclcl}
F^{\star}(D^2 \mathfrak{h}, x) &=& 0& \mbox{in} & \mathrm{B}^+_1,\\
\mathcal{G}(D\mathfrak{h},\mathfrak{h},x)  &=& g_0(x) & \mbox{on} &\mathrm{T}_1 
\end{array}
\right.
$$
belongs to $C^{2, \rho}(\mathrm{B}^{+}_1) \cap C^0(\overline{\mathrm{B}^+_1})$ and satisfies
$$
\|\mathfrak{h}\|_{C^{2, \rho}\left(\overline{\mathrm{B}^+_{r}}\right)} \le\mathrm{c}^{\ast}r^{-(2+\rho)} \left(\|\mathfrak{h}\|_{L^{\infty}(\mathrm{B}^+_1)} + \|g_0\|_{C^{1,\rho}(\overline{\mathrm{T}_1})}\right)\,\,\, \forall \,\,0<r \ll 1.
$$
\end{enumerate}

\begin{theorem}[\bf Weighted Orlicz BMO Regularity for the Hessian]\label{T2}
Assume that the structural conditions {\bf$(A_1)-(A_3)$}, $(A_4)^\star$ and $(A_5)$ hold. Let $u$ be a $L^{p}$-viscosity solution of \eqref{1.1} with $p=\bar{p}n$ for the constant $\bar{p}>1$ of the Lemma \ref{mergulhoorliczlebesgue}, $f\in L^{\Upsilon}_{\omega}-\mathrm{BMO}(\Omega)$ for $\Upsilon(t)=\Phi(t^{n})$. Then, $D^{2}u\in L^{\Upsilon}_{\omega}-\mathrm{BMO}(\Omega)$ with the following estimate
\begin{equation*}
\|D^{2}u\|_{L^{\Upsilon}_{\omega}-\mathrm{BMO}(\Omega)}\leq \mathrm{C}(\|f\|_{L^{\Upsilon}_{\omega}-\mathrm{BMO}(\Omega)}+\|\varphi_{u}\|_{L^{\Upsilon}_{\omega}-\mathrm{BMO}(\Omega)}+\|g\|_{C^{1,\alpha}(\partial\Omega)}),
\end{equation*}
for a positive constant $\mathrm{C}$ that depends only on $n$, $\lambda$, $\Lambda$, {$a$, $b$,} $\bar{p}$, {$p_{2}$}, $\Phi$, $\omega$, $\delta_{0}$, $\bar{\theta}$, $\alpha^{\prime}$, $\|\beta\|_{C^{2}(\partial \Omega)}$, $\|\gamma\|_{C^{2}(\partial \Omega)}$ and $\|\partial \Omega\|_{C^{2,\alpha}}$.
\end{theorem}

%\subsubsection*{Challenges Concerning Weighted Orlicz–Sobolev Estimates for the Model \eqref{1.1}}

\paragraph{Overview of the approach.} The approach employed to prove Theorems \ref{T1} and \ref{T2} differs from the ideas commonly used for problems of the form \eqref{obliqprob}, since the right-hand side of the equation depends on the solution itself. This prevents us, at first, from directly applying results for models without such a dependence. Inspired by \cite{Bessa24} and \cite{CaffToma21} (see also \cite{BJ1} and \cite{ZJMZ}) we overcome this difficulty by means of a penalization method combined with topological arguments. 
More specifically, in order to prove Theorem \ref{T1}  we consider a penalization of problem \eqref{1.1}, in which the right-hand side does not depend on the term \(u\).
 For each penalization parameter, we ensure the existence of a unique solution to the penalized problem in the weighted Orlicz–Sobolev space, with estimates stable in this parameter, obtained via a topological approach (see Proposition \ref{penalyzedproblem}).
 Finally, through a stability process, we show that the penalized family of solutions converges to a function that solves problem \eqref{1.1} and enjoys the desired regularity estimates.

The ideas used to prove Theorem \ref{T2} are similar, replacing the estimates in Proposition \ref{penalyzedproblem} with those in  Corollary \ref{cor:hes-BMO}, doing the same penalization process.

\paragraph{Organization of the paper.}
The remainder of this manuscript is organized as follows. In Section~\ref{Section2}, we introduce the preliminary results concerning viscosity solutions and the functional spaces. In Section~\ref{Section3}, we present a proof of the weighted Orlicz--Sobolev regularity (Theorem~\ref{T1}). Section~\ref{Section4} is devoted to applications, where we derive global Morrey-type estimates. Finally, in Section~\ref{Section5} we prove the global $L^{\Upsilon}_{\omega}$--BMO regularity for the Hessian (Theorem \ref{T2}) and some consequences.

\section{Preliminary Framework}\label{Section2}

In this section, we  recall foundational definitions and preliminary results related to the viscosity solutions of the problem \eqref{1.1}, along with the important properties of the weighted Orlicz spaces.
\subsection{Viscosity Solutions}

We introduce the appropriate notions of viscosity solutions to 
\begin{equation}\label{2.1}
\left\{
\begin{array}{rclcl}
F(D^2u,Du,u,x) &=& h(x,u)& \mbox{in} &   \Omega \\
\mathcal{G}(Du,u,x)&=& g(x) &\mbox{on}& \partial\Omega,
\end{array}
\right.
\end{equation}
where $h:\Omega\times \mathbb{R}\to \mathbb{R}$ is a measurable function.
\begin{definition}[{\bf$L^{p}$-viscosity Solution}]\label{VSLp}
Let $F$ be a $(\lambda,\Lambda,a,b)$-elliptic operator, $p>\frac{n}{2}$ and $h\in L^{p}(\Omega\times \mathbb{R})$. Assume that $F$ is continuous in the first three variables, and measurable in $x$. A function $u\in C^{0}(\overline{\Omega})$ is said to be a $L^{p}$-viscosity solution of $\eqref{2.1}$ if the following assertions hold:
\begin{enumerate}
\item [a)] $u$ is a $L^{p}$-viscosity solution to equation
\[
F(D^{2}u,Du,u,x)=h(x,u)\quad\text{in}\quad \Omega.
\]
Precisely, for all $\phi\in W^{2,p}(\Omega)$ touching $u$ by above (resp. below)  at  $x_0 \in \Omega$
$$
F\big(D^2 \phi(x_{0}), D \phi(x_{0}), \phi(x_{0}), x_{0}\big)  \geq(\leq) h\big(x_0,\phi(x_{0})\big).
$$
\item [b)] $u$ satisfies the boundary condition in the viscosity sense, namely
\[
\mathcal{G}(Du,u,x)=g(x)\quad\text{on}\quad \partial\Omega.
\]
More precisely, for every $x_{0}\in \partial\Omega$ and every test function
$\phi\in C^{1}(\overline{\Omega})$ touching $u$ from above (resp. from below) at $x_{0}$ in the sense that
\[
u(x)\le \phi(x)\quad (\text{resp. }u(x)\ge \phi(x))
\quad\text{for all }x\in \overline{\Omega}\cap B_{r}(x_{0})
\]
for some $r>0$, we have
\[
\mathcal{G}\big(D\phi(x_{0}),\phi(x_{0}),x_{0}\big)\ge g(x_{0})
\quad
(\text{resp. }\le g(x_{0})).
\]
\end{enumerate}
\end{definition}
\begin{remark}
It is worth pointing out that Definition \ref{VSLp} is understood in a ``system'' sense.
More precisely, we require that $u$ satisfies the equation
\[
F(D^{2}u,Du,u,x)=h(x,u)
\quad\text{in}\quad \Omega
\]
in the $L^{p}$-viscosity sense, while the boundary condition
\[
\mathcal{G}(Du,u,x)=g(x)
\quad\text{on}\quad \partial\Omega
\]
is imposed in the classical viscosity sense, that is, by testing with smooth functions
touching $u$ at boundary points in the relative topology of $\overline{\Omega}$.
This is consistent with the fact that, in our results, we assume $p>n$, and thus
the Sobolev embedding yields $W^{2,p}\hookrightarrow C^{1,\alpha}$, for
$\alpha=1-\frac{n}{p}$, so that the gradient $D\phi(x_{0})$ is well-defined pointwise
and the expression $\mathcal{G}(D\phi(x_{0}),\phi(x_{0}),x_{0})$ is meaningful for any
boundary contact point $x_{0}\in \partial\Omega$.
Moreover, we refer the readers to \cite{CC,CCKS,Lieberman} for a more detailed exposition of the notion and properties on viscosity solutions.
\end{remark}
We also state 
 the stability result for viscosity solutions, whose proof can be found in  \cite[Theorem 3.8]{CCKS}. Such a result will be needed \textit{a posteriori}.
\begin{lemma}[{\bf Stability Lemma}]\label{stability}
For $k \in \mathbb{N}$, let $\Omega_k \subset \Omega_{k+1}$ be an increasing sequence of domains and $\displaystyle \Omega \defeq \bigcup_{k=1}^{\infty} \Omega_k$, $p > n$ and $F, F_k$ be $(\lambda, \Lambda, a, b)-$elliptic operators. Assume $f \in L^{p}(\Omega)$, $f_k \in L^p(\Omega_k)$, and that $u_k \in C^0(\Omega_k)$ is a $L^{p}-$viscosity sub-solutions (resp. super-solutions) of
$$
	F_k(D^2 u_k,Du_k,u_k,x)=f_k(x) \quad \textrm{in} \quad \Omega_k.
$$
Suppose that $u_k \to u_{\infty}$ locally uniformly in $\Omega$ and that for $\mathrm{B}_r(x_0) \subset \Omega$ and $\phi \in W^{2,p}(\mathrm{B}_r(x_0))$ we have
\begin{equation} \label{Est1}
	\|(\hat{g}-\hat{g}_k)^+\|_{L^p(\mathrm{B}_r(x_0))} \to 0 \quad \left(\textrm{resp.} \,\,\, \|(\hat{g}-\hat{g}_k)^-\|_{L^p(\mathrm{B}_r(x_0))} \to 0 \right),
\end{equation}
where $\hat{g}(x) \defeq F(D^2 \phi, D \phi, u,x)-f(x)$ and $\hat{g}_k(x) =  F_k(D^2 \phi, D \phi, u_{k},x)-f_k(x)$.  Then, $u$ is an $L^{p}-$viscosity sub-solution (resp. super-solution) of
$$
	F(D^2 u,Du,u,x)=f(x) \quad \textrm{in} \quad \Omega.
$$
In addition, if $F$ and $f$ are continuous,  $u$ is a $C^0-$viscosity sub-solution (resp. super-solution), provided that \eqref{Est1} holds for every  test function $\phi \in C^2(\mathrm{B}_r(x_0))$.
\end{lemma}

\subsection{Weighted Orlicz and BMO Spaces}\label{wos}

%Here we introduce the framework used throughout this manuscript, namely, the \textit{weighted Orlicz spaces}.

In order to introduce the \textit{weighted Orlicz spaces}, we begin by recalling the concept of an $N-$function.

\begin{definition}[{\bf $N$-function}]
A function $\Phi:[0,+\infty) \to [0,+\infty)$ is called an $N-$function if it is convex, increasing, and continuous, satisfies $\Phi(0)=0$, $\Phi(t)>0$ for all $t>0$, and
\begin{eqnarray*}
\lim_{t \to 0^{+}}\frac{\Phi(t)}{t}=0 \quad \text{and} \quad \lim_{t\to+\infty}\frac{\Phi(t)}{t}=+\infty.
\end{eqnarray*}
\end{definition}
In this context, we say that an $N-$function $\Phi$ satisfies the $\Delta_{2}$-condition (respectively, the $\nabla_{2}$-condition) if there exists a constant $\mathrm{C}_{1}>1$ (respectively, $\mathrm{C}_{2}>1$) such that
\begin{eqnarray*}
\Phi(2t)\leq \mathrm{C}_{1}\Phi(t) \quad \left(\text{respectively, } \Phi(t)\leq \frac{1}{2\mathrm{C}_{2}}\Phi(\mathrm{C}_{2}t)\right), \quad  t>0.
\end{eqnarray*}
We write $\Phi \in \Delta_{2}$ (respectively, $\Phi \in \nabla_{2}$) to denote that $\Phi$ satisfies the $\Delta_{2}$-condition (respectively, the $\nabla_{2}$-condition). If $\Phi$ satisfies both, we write $\Phi \in \Delta_{2} \cap \nabla_{2}$.
Given a $\Phi \in \Delta_{2} \cap \nabla_{2}$, we can define its \textit{lower index} by
\begin{eqnarray*}
i(\Phi)=\lim_{t \to 0^{+}}\frac{\log(\mathfrak{s}_{\Phi}(t))}{\log t} = \sup_{0<t<1}\frac{\log(\mathfrak{s}_{\Phi}(t))}{\log t},
\
\mbox{ where }
\
\mathfrak{s}_{\Phi}(t)=\sup_{s>0}\frac{\Phi(st)}{\Phi(s)}, \quad t>0.
\end{eqnarray*}
As illustrative examples, the functions $\Phi(t)=t^{q}$ and $\overline{\Phi}(t)=t^{q}\log(t+1)$, with $q>1$, are $N-$functions that satisfy both the $\Delta_{2}$- and $\nabla_{2}$-conditions, with $i(\Phi)=i(\overline{\Phi})=q>1$.\\

A function $\omega$ is called a \textit{weight} if it is nonnegative, locally integrable, and positive almost everywhere. In this case, we identify $\omega$ with the measure
\begin{eqnarray*}
\omega(\mathrm{U})=\int_{\mathrm{U}}\omega(x)\,dx,
\end{eqnarray*}
for every Lebesgue measurable set $U \subset \mathbb{R}^{n}$.
\begin{definition}[{\bf Muckenhoupt Class}]
We say that a weight $\omega$ belongs to the \textit{Muckenhoupt class} $\mathfrak{A}_{q}$, for some $q \in (1,\infty)$, and write $\omega \in \mathfrak{A}_{q}$, if
\begin{eqnarray*}
[\omega]_{q} \defeq \sup_{\mathrm{B} \subset \mathbb{R}^{n}} \left( \intav{B} \omega(x)\,dx \right) \left( \intav{B} \omega(x)^{\frac{-1}{q-1}}\,dx \right)^{q-1} < +\infty,
\end{eqnarray*}
where the supremum is taken over all balls $\mathrm{B} \subset \mathbb{R}^{n}$.
\end{definition}
\begin{remark}
The concept of the class of $\mathfrak{A}_{q}$-weights was introduced by Muckenhoupt in the mid-1970s in \cite{Muckenhoupt}. They characterize the class of weights for which classical operators—such as the Hardy–Littlewood maximal operator and singular integral operators—are bounded on weighted Lebesgue, Lorentz and Orlicz spaces. Such \(\mathfrak{A}_{q}\) weights are essential tools in the study of regularity theory, particularly in the context of elliptic and parabolic equations with nonuniform degeneracies or singularities, see \cite{KokiMiro} for more details.
\end{remark}

Finally, we are able to present the definition of the main functional space in this work.
\begin{definition}[{\bf Weighted Orlicz Spaces}]\label{weightedorliczspaces}
The \textit{weighted Orlicz space} $L^{\Phi}_{\omega}(\mathrm{U})$ for an $N-$function $\Phi\in \Delta_{2}\cap\nabla_{2}$, a Lebesgue measurable set $\mathrm{U}\subset \mathbb{R}^{n}$ and a weight $\omega$ is the space of all measurable functions $h$, defined over $\mathrm{U}$, such that
\begin{eqnarray*}
\rho_{\Phi,\omega}(h)=:\int_{\mathrm{U}}\Phi(|h(x)|)\omega(x)dx<+\infty.
\end{eqnarray*}
Due to the condition $\Phi\in\Delta_{2}\cap\nabla_{2}$, the \textit{weighted Orlicz space} $L^{\Phi}_{\omega}(\mathrm{U})$ is a reflexive Banach space, when endowed with the following Luxemburg norm
\begin{eqnarray*}
\|h\|_{L^{\Phi}_{\omega}(\mathrm{U})}=:\inf\left\{t>0:\rho_{\Phi,\omega}\left(\frac{h}{t}\right)\leq 1\right\}.
\end{eqnarray*}
The \textit{weighted Orlicz-Sobolev space} $W^{k,\Phi}_{\omega}(\mathrm{U})$, for an integer $k\geq0$, is the set of all measurable functions $h$ over $\mathrm{U}$ such that all derivatives $D^{\sigma}h$, in distributional sense, for any multiindex $\sigma$ with length $|\sigma|=0, 1,\cdots,k$  belong to $L^{\Phi}_{\omega}(\mathrm{U})$, whose norm is given by
\begin{eqnarray*}
\|h\|_{W^{k,\Phi}_{\omega}(\mathrm{U})}\defeq \sum_{|\sigma|\leq k }\|D^{\sigma}h\|_{L^{\Phi}_{\omega}(\mathrm{U})}.
\end{eqnarray*} 
\end{definition}

Notice that if $\Phi(t) = t^{p}$ for some $p \in (1, +\infty)$, then $L^{\Phi}_{\omega}(\mathrm{U})$ corresponds to the weighted Lebesgue space $L^{p}_{\omega}(\mathrm{U})$, and $W^{k,\Phi}_{\omega}(\mathrm{U})$ corresponds to the weighted Sobolev space $W^{k,p}_{\omega}(\mathrm{U})$, recovering the classical Lebesgue and Sobolev spaces when $\omega = 1$.
\begin{remark}
For an $N-$function $\Phi \in \Delta_{2} \cap \nabla_{2}$, there exist constants $p_{1}, p_{2}$ with $1 < p_{1} \leq p_{2} < +\infty$ such that
\begin{eqnarray}\label{propriedadedei(Phi)}
\bar{\mathrm{C}}^{-1}\min\{s^{p_{1}},s^{p_{2}}\}\Phi(t)\leq\Phi(st)\leq \bar{\mathrm{C}}\max\{s^{p_{1}},s^{p_{2}}\}\Phi(t), \ \forall t,s\geq 0,
\end{eqnarray}
where $\bar{\mathrm{C}} > 0$ is independent of $t$ and $s$, see \cite{KokiMiro}. Consequently, the following inclusions hold:
$$
L^{\infty}(\mathrm{U}) \subset L^{p_{2}}_{\omega}(\mathrm{U}) \subset L^{\Phi}_{\omega}(\mathrm{U}) \subset L^{p_{1}}_{\omega}(\mathrm{U}) \subset L^{1}(\mathrm{U}).
$$
Moreover, in view of the inequality \eqref{propriedadedei(Phi)}, the quantity $i(\Phi)$ can be characterized as the supremum of all constants $p_{1}$ such that it remains valid for every $s \geq 1$ (see \cite{fioreza} for further details). Consequently, we obtain $i(\Phi) > 1$, so that it is appropriate to consider the Muckenhoupt class $\mathfrak{A}_{i(\Phi)}$.
\end{remark}
Next, we introduce a useful property, needed later, that connects the norm with the modular in weighted Orlicz spaces on the unit ball; see \cite[Lemma 2.1.14]{DHHR11}.

\begin{lemma}[\bf Norm-modular unit ball property]\label{normmodular}
In the weighted Orlicz space \(L^{\Phi}_{\omega}(\mathrm{U})\),  the following equivalence holds frue:
\begin{equation*}
\rho_{\Phi,\omega}(h)\leq 1\Longleftrightarrow \|h\|_{L^{\Phi}_{\omega}(\mathrm{U})}\leq 1.
\end{equation*}
\end{lemma}

Finally, the key point arising from the preceding definitions and results is that $L^{\Phi}_{\omega}(\Omega)$ admits a continuous embedding into a suitable Lebesgue space.
\begin{lemma}[{\bf \cite[Lemma 5]{BLOK}}]\label{mergulhoorliczlebesgue}
Let $\Phi$ be an $N-$function such that $\Phi\in\Delta_{2}\cap\nabla_{2}$, $\omega\in \mathfrak{A}_{i(\Phi)},$ and $\Omega\subset \mathbb{R}^{n}$ be a bounded domain.  Then, there exists $\bar{p}\in (1,i(\Phi))$ depending only on $i(\Phi)$ and $\omega$  such that $L^{\Phi}_{\omega}(\Omega)$ is continuously embedded in $L^{\bar{p}}(\Omega),$ and the following estimate holds
\begin{eqnarray*}
\|h\|_{L^{\bar{p}}(\Omega)}\leq \mathrm{C}\|h\|_{L^{\Phi}_{\omega}(\Omega)}, \ \forall \, h\in L^{\Phi}_{\omega}(\Omega),
\end{eqnarray*}
where $\mathrm{C}$ is a positive constant that depends only on \(n\), \(i(\Phi)\) and \(\omega\).
\end{lemma}
From now on, we introduce the $BMO$-spaces, we begin defining the functions of bounded mean oscillation, $BMO$-functions, for short.
\begin{definition}
 A function $f \in L^1_{loc}(\Omega)$ is said to be a  $p-BMO$ function
 if
\[ \|f\|_{p-BMO(\Omega)} := \sup_{B\subset\Omega}
\left(\int_{B\cap\Omega}|f(x) - f_B|^p dx\right)^{\frac 1p}< +\infty\]
for every  ball $B \subset\Omega$, where
\[f_B:=\frac{1}{|B|}\int_B
f(x)dx.\] 
\end{definition}
As a consequence of the John–Nirenberg inequality,
such a semi-norm is equivalent to the one in the classical BMO spaces (see \cite[pp. 763-764]{pimentel}).
\begin{definition}[{\bf Weighted Orlicz-Sobolev BMO Spaces}]
A function $f \in L^1_{loc}(\Omega)$ is said to belong to the space $L^\Phi_\omega -\mathrm{BMO}(\Omega)$
 for $\Phi \in \Delta_2 \cap \nabla_2$, an $N$-function and a weight $\omega\in \mathfrak{A}_{i(\Phi)}$ if
\[\|f\|_{L^\Phi_\omega -\mathrm{BMO}(\Omega)} := \sup_{B\subset\Omega}
\frac{\|(f - f_B)\chi_B\|_{L^\phi_\omega(\Omega)}}
{\|\chi_B\|_{L^\Phi_\omega(\Omega)}}
< +\infty.\]
In particular, if $\Phi(t) = t^p$ for $p > 1$, we achieve the definition of the $p-\mathrm{BMO}$ spaces.
\end{definition}
\begin{remark}\label{Remark2.12}
Using the fact that $\Phi \in \Delta_2 \cap \nabla_2$ and assuming $\omega \in \mathfrak{A}_{i(\Phi)}$, it  follows
from \cite{LLO} that there exist  universal constants $0 < \mathrm{C}_1 \leq \mathrm{C}_2$ such that
\[\mathrm{C}_1\|f\|_{\mathrm{BMO}(\Omega)} \leq \|f\|_{L^{\Phi}_{\omega}-\mathrm{BMO}(\Omega)} \leq \mathrm{C}_2\|f\|_{\mathrm{BMO}(\Omega)},\ \mbox{for all } f \in L^1_{loc}(\Omega).\]
\end{remark}
%%%%%%%%%%%%%%%%%%%%%%%%%%%%%%%%%%%%%%%%%%%%%%%%%%%%%%%%%%%%%%%%%%%%%
\section{Weighted Orlicz-Sobolev Estimates}\label{Section3}

In this section, we present the proof of Theorem \ref{T1}. We begin by examining the existence of solutions to the associated penalized problem. For this purpose, we state the following result.
\begin{proposition}\label{penalyzedproblem}
Assume the structural conditions {\bf$(A_1)$}-{\bf $(A_5)$} and {\bf $(E_1)$}-{\bf$(E_2)$} are in force. For each \(\varepsilon>0\), there exists  \(u_{\varepsilon}\in W^{2,\Upsilon}_{\omega}(\Omega)\), for \(\Upsilon(t)=\Phi(t^{n})\), a viscosity solution of 
\begin{equation}\label{penprob}
\left\{
\begin{array}{rclcl}
 F(D^2u_{\varepsilon},Du_{\varepsilon},u_{\varepsilon},x) &=& f(x)+\varphi_{u_{\varepsilon},\varepsilon}(x)& \mbox{in} &   \Omega \\
\mathcal{G}(Du_{\varepsilon},u_{\varepsilon},x) &=& g(x) &\mbox{on}& \partial \Omega,
\end{array}
\right.
\end{equation}
where \(\varphi_{u,\varepsilon}\) is defined by
\[
\varphi_{u_{\varepsilon},\varepsilon}(x)=\varphi\left(\frac{1}{\varepsilon}\int_{0}^{\varepsilon}|\{z\in \Omega\ |\ u_{\varepsilon}(z)\geq u_{\varepsilon}(x)-t\}|dt\right).
\]
Furthermore, the following estimate holds
\[
\|u_{\varepsilon}\|_{W^{2,\Upsilon}_{\omega}(\Omega)}\leq \mathrm{C}(\|f\|_{L^{\Upsilon}_{\omega}(\Omega)}+\|\varphi_{u_{\varepsilon},\varepsilon}\|_{L^{\Upsilon}_{\omega}(\Omega)}+\|g\|_{C^{1,\alpha}(\partial \Omega)}),
\]
where \(\mathrm{C}>0\) depends only on \(n\), \(\lambda\), \(\Lambda\), \(a\), \(b\), \(\delta_{0}\), \(\bar{p}\), \(p_{2}\), \(\Phi\), \(\omega\), \(\bar{\theta}\), \(\alpha^{\prime}\), \(\|\beta\|_{C^{2}(\partial\Omega)}\), \(\|\gamma\|_{C^{2}(\partial \Omega)}\), \(\partial\Omega\) and \(\operatorname{diam}(\Omega)\).
\end{proposition}
\begin{proof}
We proceed by applying a penalization method. Given $v \in C^{0,1}(\Omega)$, we note that the function $\varphi_{v,\varepsilon}$ is defined by
\[
\varphi_{v,\varepsilon}(x)=\varphi\left(\frac{1}{\varepsilon}\int_{0}^{\varepsilon}|\{z\in \Omega\ |\ v(z)\geq v(x)-t\}|dt\right),
\]
and belongs to \(L^{\infty}(\Omega)\) satisfying 
\(
\|\varphi_{v,\varepsilon}\|_{L^{\infty}(\Omega)}\leq \|\varphi\|_{L^{\infty}([0,|\Omega|])}.
\)
Consequently, we have that \(\varphi_{v,\varepsilon}\in L^{\Upsilon}_{\omega}(\Omega)\) and
\begin{eqnarray}\label{est1prop3.1}
\int_{\Omega}\Upsilon(|\varphi_{v,\varepsilon}(x)|)\omega(x)dx\leq \Upsilon(\|\varphi\|_{L^{\infty}([0,|\Omega|])})\omega(\Omega).
\end{eqnarray}
Under the structural assumptions, we may apply Perron's method \cite[Theorem 7.19]{Lieberman}, combined with the global weighted Orlicz estimates for fully nonlinear models with oblique boundary condition  \cite[Theorem 3.5]{Bessa24}, to ensure the existence of a unique solution $v_{\varepsilon} \in W^{2,\Upsilon}_{\omega}(\Omega)$ of
\begin{equation*}
\left\{
\begin{array}{rclcl}
 F(D^2v_{\varepsilon},Dv_{\varepsilon},v_{\varepsilon},x) &=& f(x)+\varphi_{v,\varepsilon}(x)& \mbox{in} &   \Omega \\
\mathcal{G}(Dv_{\varepsilon},v_{\varepsilon},x) &=& g(x) &\mbox{on}& \partial \Omega,
\end{array}
\right.
\end{equation*}
satisfying the global estimate
\begin{eqnarray}\label{est2prop3.1}
\|v_{\varepsilon}\|_{W^{2,\Upsilon}_{\omega}(\Omega)}\leq \mathrm{C}(\|f\|_{L^{\Upsilon}_{\omega}(\Omega)}+\|\varphi_{v,\varepsilon}\|_{L^{\Upsilon}_{\omega}(\Omega)}+\|g\|_{C^{1,\alpha}(\partial \Omega)}),
\end{eqnarray}
where \(\mathrm{C}>0\) depends only on \(n\), \(\lambda\), \(\Lambda\), \(a\), \(b\), \(\delta_{0}\), \(\bar{p}\), \(p_{2}\), \(\Phi\), \(\omega\), \(\bar{\theta}\), \(\alpha^{\prime}\), \(\|\beta\|_{C^{2}(\partial\Omega)}\), \(\|\gamma\|_{C^{2}(\partial \Omega)}\), \(\partial \Omega \) and \(\operatorname{diam}(\Omega)\). 

In order to estimate the norm \(\|\varphi_{v,\varepsilon}\|_{L^{\Upsilon}_{\omega}(\Omega)}\), since \(\rho_{\Upsilon,\omega}\) is a modular, it follows that
\[
\rho_{\Upsilon,\omega}\left(\frac{|\varphi_{v,\varepsilon}|}{1+\rho_{\Upsilon,\omega}(|\varphi_{v,\varepsilon}|)}\right)\leq \frac{\rho_{\Upsilon,\omega}(|\varphi_{v,\varepsilon}|)}{1+\rho_{\Upsilon,\omega}(|\varphi_{v,\varepsilon}|)} \leq 1.
\]
So, by the norm-modular unit ball property Lemma \ref{normmodular}, and the positivity homogeneity of the norm \(\|\cdot\|_{L^{\Upsilon}_{\omega}(\Omega)}\) we may conclude that
\begin{eqnarray}\label{vphiest}
\left\|\frac{\varphi_{v,\varepsilon}}{1+\rho_{\Upsilon,\omega}(\varphi_{v,\varepsilon})}\right\|_{L^{\Upsilon}_{\omega}(\Omega)}\leq 1 \Longrightarrow \|\varphi_{v,\varepsilon}\|_{L^{\Upsilon}_{\omega}(\Omega)}\leq 1+\rho_{\Upsilon,\omega}(\varphi_{v,\varepsilon})\stackrel{\eqref{est1prop3.1}}{\leq} 1+\Phi(\|\varphi\|_{L^{\infty}([0,|\Omega|])}^{n})\omega(\Omega).
\end{eqnarray}
Therefore, by using \eqref{vphiest} in the estimate \eqref{est2prop3.1}, we deduce that\begin{eqnarray}\label{est3prop3.1}
\|v_{\varepsilon}\|_{W^{2,\Upsilon}_{\omega}(\Omega)}\leq \mathrm{C}(1+\|f\|_{L^{\Upsilon}_{\omega}(\Omega)}+\Phi(\|\varphi\|_{L^{\infty}([0,|\Omega|])}^{n})\omega(\Omega)+\|g\|_{C^{1,\alpha}(\partial \Omega)})\defeq \mathrm{C}_{0}.
\end{eqnarray}
In other words, the operator $\mathcal{T}: C^{0,1}(\Omega) \to W^{2,\Upsilon}_{\omega}(\Omega) \subset C^{0,1}(\Omega)$, given by $\mathcal{T}v = v_{\varepsilon}$, is well defined and maps closed balls into themselves in the $C^{0,1}(\Omega)$-topology. In fact,  the $C^{0,1}(\Omega)$-norm can be estimated by the $W^{2,\Upsilon}_{\omega}(\Omega)$-norm via the embedding given by Lemma \ref{mergulhoorliczlebesgue}, combined with the Sobolev embedding, resulting the following chain of embeddings
\[W^{2,\Upsilon}_{\omega}(\Omega)\hookrightarrow W^{2,p}(\Omega)\hookrightarrow C^{1,1-\frac{n}{p}}(\Omega)\hookrightarrow C^{0,1}(\Omega). \]
Thus, \(\mathcal{T}\) is a compact operator, and using the Schauder's fixed point theorem, cf. \cite[pp. 179]{Brezis}, there exists a function \(u_{\varepsilon}\in W^{2,\Upsilon}_{\omega}(\Omega)\) such that \(\mathcal{T}u_{\varepsilon}=u_{\varepsilon}\), i.e., \(u_{\varepsilon}\) is a solution of \eqref{penprob} and satisfies the estimate \eqref{est2prop3.1}.
\end{proof}

We are now able to prove Theorem \ref{T1}.
\begin{proof}[{\bf Proof of Theorem \ref{T1}}]
Initially, from Proposition \ref{penalyzedproblem}, for each \(\varepsilon>0\), we can consider \(u_{\varepsilon}\in W^{2,\Upsilon}_{\omega}(\Omega)\), which is a solution of
\begin{equation*}
\left\{
\begin{array}{rclcl}
 F(D^2u_{\varepsilon},Du_{\varepsilon},u_{\varepsilon},x) &=& f(x)+\varphi_{u,\varepsilon}(x)& \mbox{in} &   \Omega \\
\mathcal{G}(Du_{\varepsilon},u_{\varepsilon},x) &=& g(x) &\mbox{on}& \partial \Omega,
\end{array}
\right.
\end{equation*}
where \(u_{\varepsilon}\) satisfies the following estimate
\begin{equation}\label{est1teo1.2}
\|u_{\varepsilon}\|_{W^{2,\Upsilon}_{\omega}(\Omega)}\leq \mathrm{C}(\|f\|_{L^{\Upsilon}_{\omega}(\Omega)}+\|\varphi_{u_{\varepsilon},\varepsilon}\|_{L^{\Upsilon}_{\omega}(\Omega)}+\|g\|_{C^{1,\alpha}(\partial \Omega)}),
\end{equation}
for a universal constant \(\mathrm{C}>0\). As in Proposition \ref{penalyzedproblem}, we can estimate the term $\|\varphi_{u_{\varepsilon},\varepsilon}\|_{L^{\Upsilon}_{\omega}(\Omega)}$ and obtain the following boundedness
\begin{eqnarray*}
\|u_{\varepsilon}\|_{W^{2,\Upsilon}_{\omega}(\Omega)}\leq \mathrm{C}(1+\|f\|_{L^{\Upsilon}_{\omega}(\Omega)}+\Phi(\|\varphi\|_{L^{\infty}([0,|\Omega|])}^{n})\omega(\Omega)+\|g\|_{C^{1,\alpha}(\partial \Omega)}) = \mathrm{C}_{0},
\end{eqnarray*}
where $\mathrm{C}_{0}>0$, as defined in \eqref{est3prop3.1}, does not depend on $\varepsilon$ and $u_{\varepsilon}$. In other words, the family $(u_{\varepsilon})_{\varepsilon>0}$ is bounded in $W^{2,\Upsilon}_{\omega}(\Omega)$, and since this space is reflexive, there exist a subsequence $(u_{\varepsilon_{j}})_{j\in\mathbb{N}}$ with $\varepsilon_{j}\to0$ as $j\to \infty$, and a function $u_{\infty}\in W^{2,\Upsilon}_{\omega}(\Omega)$ such that $u_{\varepsilon_{j}}\rightharpoonup u_{\infty}$ weakly in $W^{2,\Upsilon}_{\omega}(\Omega)$. Furthermore, by applying the suitable Sobolev’s embedding together with Lemma \ref{mergulhoorliczlebesgue}, we obtain, possibly after passing to a subsequence, that $u_{\varepsilon_{j}} \to u_{\infty}$ in $C^{1,1-\frac{n}{p}}(\overline{\Omega})$. In addition, due to this convergence, the oblique boundary condition for $u_{\infty}$ is satisfied pointwisely, namely, 
\[
\mathcal{G}(Du_{\varepsilon_{j}},u_{\varepsilon_{j}},x)\to \mathcal{G}(Du_{\infty},u_{\infty},x)\,\, \text{on}\,\, \partial \Omega.
\]
Thus, to apply the stability result given by Lemma \ref{stability}, it remains to verify that 
$$\varphi_{u_{\varepsilon_{j}},\varepsilon_{j}} \to \varphi(|\{z \in \Omega \mid u_{\infty}(z) \geq u_{\infty}(x)\}|) \defeq \varphi_{u_{\infty}}(x)\,\, \text{ in }\,\, L^{p}(\Omega)\,\, \text{ as }\,\,j \to \infty.$$
In effect, by the uniform convergence \(u_{\varepsilon_{j}}\to u_{\infty}\) in \(C^{1,\alpha}\)- norm, then we may apply the Lebesgue differentiation Theorem, cf. \cite[pp. 104]{SteinShak05}, to conclude that
\(
\varphi_{u_{\varepsilon_{j}},\varepsilon_{j}}(x) \to  \varphi_{u_{\infty}}(x) \,\, \text{in }\,\, \Omega.
\)
Moreover, using the uniform continuity of $\varphi$, and applying the Lebesgue dominated convergence theorem  \cite [pp. 67]{SteinShak05},  the desired convergence $\varphi_{u_{\varepsilon_{j}},\varepsilon_{j}} \to \varphi_{u_{\infty}}$ in the $L^{p}$-norm follows. 

Therefore, we are able to apply the stability result and conclude that $u = u_{\infty}$ is a viscosity solution of
\begin{equation*}
\left\{
\begin{array}{rclcl}
 F(D^2u,Du,u,x) &=& f(x)+\varphi_{u}(x)& \mbox{in} &   \Omega \\
\mathcal{G}(Du,u,x) &=& g(x) &\mbox{on}& \partial \Omega.
\end{array}
\right.
\end{equation*}
The weak convergence $u_{\varepsilon_{j}} \rightharpoonup u$ in $W^{2,\Upsilon}_{\omega}(\Omega)$ and the estimate \eqref{est1teo1.2} imply that
\begin{equation}\label{est2teo1.2}
\|u\|_{W^{2,\Upsilon}_{\omega}(\Omega)}\leq \liminf_{j\to \infty}\|u_{\varepsilon_{j}}\|_{W^{2,\Upsilon}_{\omega}(\Omega)}\leq \mathrm{C}(\|f\|_{L^{\Upsilon}_{\omega}(\Omega)}+\liminf_{j\to \infty}\|\varphi_{u_{\varepsilon_{j}},\varepsilon_{j}}\|_{L^{\Upsilon}_{\omega}(\Omega)}+\|g\|_{C^{1,\alpha}(\partial \Omega)}).
\end{equation}
Finally, applying  the same argument used to establish the convergence $\varphi_{u_{\varepsilon_{j}},\varepsilon_{j}} \to \varphi_{u}$ in  $L^{p}(\Omega)$, we obtain the convergence of the modular $\rho_{\Upsilon,\omega}(\varphi_{u_{\varepsilon_{j}},\varepsilon_{j}} - \varphi_{u}) \to 0$, which implies that $\varphi_{u_{\varepsilon_{j}},\varepsilon_{j}} \to \varphi_{u}$ in $L^{\Upsilon}_{\omega}(\Omega)$. Indeed, given $\eta_{0}>0$, by the convergence of the modular, it follows that there exists $j_{0} \in \mathbb{N}$ such that
$$
\int_{\Omega}\Upsilon(|\varphi_{u_{\varepsilon_{j}},\varepsilon_{j}}(x)-\varphi_{u}(x)|)\omega(x)\,dx<\mathrm{C}_{\eta_{0}}^{-1}, \quad \forall \ j\geq j_{0},
$$
where $\mathrm{C}_{\eta_{0}}=\overline{\mathrm{C}}\max\{(2^{-1}\eta_{0})^{-p_{1}n},(2^{-1}\eta_{0})^{-p_{2}n}\}$, for some $\overline{\mathrm{C}}>0$, and $1<p_{1}\leq p_{2}<+\infty$ are the constants given in \eqref{propriedadedei(Phi)}. Consequently,
$$
\displaystyle\int_{\Omega}\Upsilon\left(\frac{|\varphi_{u_{\varepsilon_{j}},\varepsilon_{j}}(x)-\varphi_{u}(x)|}{2^{-1}\eta_{0}}\right)\omega(x)\,dx \leq \mathrm{C}_{\eta_{0}}\int_{\Omega}\Upsilon(|\varphi_{u_{\varepsilon_{j}},\varepsilon_{j}}(x)-\varphi_{u}(x)|)\omega(x)\,dx < 1,
$$
for all $j\geq j_{0}$. Thus, the definition of the weighted Orlicz norm $\|\cdot\|_{L^{\Upsilon}_{\omega}(\Omega)}$ implies that

$$
\|\varphi_{u_{\varepsilon_{j}},\varepsilon_{j}} - \varphi_{u}\|_{L^{\Upsilon}_{\omega}(\Omega)} \leq \frac{\eta_{0}}{2} < \eta_{0}, \quad \forall \ j \geq j_{0},
$$
proving the desired convergence. Therefore, in \eqref{est2teo1.2} we have that
\begin{equation*}
\|u\|_{W^{2,\Upsilon}_{\omega}(\Omega)}\leq \mathrm{C}(\|f\|_{L^{\Upsilon}_{\omega}(\Omega)}+\|\varphi_{u}\|_{L^{\Upsilon}_{\omega}(\Omega)}+\|g\|_{C^{1,\alpha}(\partial \Omega)}),
\end{equation*}
 completing the proof.
\end{proof}

\begin{remark}
By taking $\varphi = 0$, we recover the weighted Orlicz estimates obtained in \cite[Theorem 3.5]{Bessa24}.
\end{remark}

Our regularity result in Theorem \ref{T1} can be applied straightway in the context of classical weighted Lebesgue spaces. Summarizing, we have the following result.
\begin{theorem}\label{weightedLqestimates}
Assume the structural conditions {\bf $(A_1)$}-{\bf $(A_5)$} are in force, considering in {\bf$(A_2)$} that \(f\) belongs only to \(L^{q}_{\omega}(\Omega)\) for \(n<q<+\infty\). Suppose that {\bf $(E_1)$}-{\bf$(E_2)$} hold true and that \(\partial\Omega\in C^{3}\). Then, there exists \(u\in W^{2,q}_{\omega}(\Omega)\), a \(L^{q}\) - viscosity solution of \eqref{1.1} satisfying the following estimate
\begin{eqnarray}\label{estLq}
\|u\|_{W^{2,q}_{\omega}(\Omega)}\leq \mathrm{C}(\|f\|_{L^{q}_{\omega}(\Omega)}+\|\varphi_{u}\|_{L^{q}_{\omega}(\Omega)}+\|g\|_{C^{1,\alpha}(\partial \Omega)}),
\end{eqnarray}
where \(\mathrm{C}>0\) is a constant that depends only on \(n\), \(\lambda\), \(\Lambda\), \(a\), \(b\), \(q\), \(\omega\), \(\delta_0\), \(\bar{\theta}\), \(\alpha^{\prime}\), \(\|\beta\|_{C^{2}(\partial \Omega)}\), \(\|\gamma\|_{C^{2}(\partial \Omega)}\) and \(\|\partial \Omega\|_{C^{2,\alpha}}\).
\end{theorem}
\begin{proof}
Consider the $N$-function $\Phi(t) = t^{\frac{q}{n}}$. Clearly, $\Phi$ satisfies the $\Delta_{2} \cap \nabla_{2}$ condition, with $i(\Phi) = \frac{q}{n}$, and we have $|f|^{n} \in L^{\Phi}_{\omega}(\Omega)$. Therefore, we can apply Theorem \ref{T1} to conclude the existence of a viscosity solution $u \in W^{2,\Upsilon}_{\omega}(\Omega)$ to problem \eqref{1.1}, where \(\Upsilon(t)=\Phi(t^{n})\), and such that $u$ is an $L^{\bar{p}n}$-viscosity solution. Moreover, the following estimate holds:
\begin{eqnarray}\label{est1doteo3.3}
\|u\|_{W^{2,\Upsilon}_{\omega}(\Omega)}\leq \mathrm{C}(\|f\|_{L^{\Upsilon}_{\omega}(\Omega)}+\|\varphi_{u}\|_{L^{\Upsilon}_{\omega}(\Omega)}+\|g\|_{C^{1,\alpha}(\partial \Omega)}).
\end{eqnarray}
Now, by the definition of $\Phi$, we have $\Upsilon(t) = t^{q}$. Hence, by the equivalence of the norms in $L^{q}_{\omega}(\Omega)$ (the usual norm and the Luxemburg norm), estimate \eqref{est1doteo3.3} implies estimate \eqref{estLq}. Finally, by the inclusion of Lebesgue spaces, we conclude that $u$ is an $L^{q}$-viscosity solution of problem \eqref{1.1}. This completes the proof.
\end{proof}

As a immediate consequence, we have the Hölder regularity for the gradient of solutions to problem \eqref{1.1} via the Sobolev embedding theorem.
\begin{corollary}
Under the conditions of Theorem \ref{weightedLqestimates} with \(\omega\equiv 1\), the solutions to problem \eqref{1.1} belong to \(C^{1,\alpha}(\overline{\Omega})\) for any \(0<\alpha<1\),  provided that \(f\in L^{q}(\Omega)\), for all \(n<q<+\infty\).
\end{corollary}

\begin{remark}
The weighted estimate obtained in Theorem \ref{weightedLqestimates} extends that derived in \cite{Bessa}, by simply taking $\omega\equiv 1$ and $\varphi=0$.
\end{remark}

\section{An Application: Morrey Estimates}\label{Section4}

In this section, we present an application of the weighted estimates established in the previous section. 
More precisely, we derive Morrey-type bounds for the solutions of problem \eqref{1.1} in the case $\gamma=g=0$. 
This provides a refined local control of the solution in terms of scale-invariant integral quantities. 

\begin{definition}
Let $\Omega\subset \mathbb{R}^{n}$ be a domain. For any \(1\leq p<+\infty\) and \(0<\theta< n\), the \textit{Morrey space} \(L^{p,\theta}(\Omega)\) is the set of measurable functions \(h:\Omega\to \mathbb R\) such  that
\[
\|h\|_{L^{p,\theta}(\Omega)}\defeq \sup_{\substack{x_{0}\in\Omega \\ 0<r\leq \operatorname{diam}(\Omega)}} \left(\frac{1}{r^{n-\theta}}\int_{\Omega(x_{0},r)}|h(x)|^{p}dx\right)^{\frac{1}{p}}<+\infty,
\] 
where \(\Omega(x_{0},r)=\Omega\cap \mathrm{B}_{r}(x_{0})\).
The \textit{Morrey–Sobolev space} $W^{k}L^{p,\theta}(\Omega)$ is defined as the set of measurable functions $h$ on $\Omega$ for which every distributional derivative $D^{\sigma}h$ belongs to $L^{p,\theta}(\Omega)$ for all multi-indices $\sigma$ with $|\sigma| = 0, \ldots, k$. In \(W^{k}L^{p,\theta}(\Omega)\) we can equip with the  norm 
\[
\|h\|_{W^{k}L^{p,\theta}(\Omega)}\defeq\sum_{|\sigma|\leq k}\|D^{\sigma}h\|_{L^{p,\theta}(\Omega)}.
\]
\end{definition}
The main idea for obtaining estimates in Morrey spaces is to establish an appropriate decay of certain integral quantities with respect to the radius. In our setting, this is accomplished through the weighted estimates proved above, which provide control over the radius-dependent terms appearing in the local estimates.

In this scenario, we have the following theorem:
\begin{theorem}[{\bf Morrey estimates}]\label{T3} 
Let \( \Omega \subset \mathbb{R}^{n} \) (\( n \geq 2 \)) be a bounded domain with \( \partial\Omega \in C^{3} \), and assume the structural conditions {\bf$(A_1)$}, {\bf ($A_3$)}-{\bf$(A_5)$}, {\bf$(E_1)$}-{\bf $(E_2)$} hold. Let \( u \) be a \(L^{q}\)- viscosity solution of \eqref{1.1}, where \( \beta \in C^{2}(\partial \Omega) \), \(\gamma=g=0\) and \( f \in L^{q,\bar\theta}(\Omega)\) for \(0<\bar\theta< n<q<+\infty\). Then,  \(u\in W^{2}L^{q,\bar\theta}(\Omega)\), with the estimate  
\begin{eqnarray}
\|u\|_{W^{2}L^{q,\bar\theta}(\Omega)}\le \mathrm{C}\cdot (\|f\|_{L^{q,\bar\theta}(\Omega)}+\|\varphi_{u}\|_{L^{q,\bar\theta}(\Omega)}),
\end{eqnarray}
where $\mathrm{C}=\mathrm{C}(n,\lambda,\Lambda,a,b,\delta_{0},\bar\theta,\|\beta\|_{C^{2}(\partial\Omega)},\Omega)$ is a positive constant.
\end{theorem}
\begin{proof}
We show the result for \( q < +\infty \).  For \( q = +\infty \), the proof follows in a   similar way with minor modifications.
First, fix   \(x_{0}\in \Omega\), \(0<r<\operatorname*{diam}(\Omega)\) and extend \(f\) by zero in $\mathbb R^n\setminus
\Omega$.   For \(\tau\in(0,\bar\theta)\), set
\[
\omega(x):=\min\{|x-x_{0}|^{-n+\bar\theta-\tau},r^{-n+\bar\theta-\tau}\}.
\]
As $\bar\theta-\tau<n$, it follows that \(\omega\) is a weight that belongs to the {\it Muckenhoupt class} \(\mathfrak{A}_{\frac{q}{n}}\) with \([\omega]_{\frac{q}{n}} \leq 1\). Since \(\omega_{|_{\mathrm{B}_{r}(x_{0})}} = r^{-n+\bar\theta-\tau}\),  using Theorem \ref{weightedLqestimates}, we obtain 
\begin{eqnarray}\label{est1teomorrey}
\|D^{2}u\|_{L^{q}(\Omega(x_{0},r))}=r^{\frac{n-\bar\theta+\tau}{q}}\|D^{2}u\|_{L^{q}_{\omega}(\Omega(x_{0},r))}\leq \mathrm{C}r^{\frac{n-\bar\theta+\tau}{q}}\left(\|f\|_{L^{q}_{\omega}(\Omega)}+\|\varphi_{u}\|_{L^{q}_{\omega}(\Omega)}\right). 
\end{eqnarray}

First, we find an estimate for the first expression in the right-hand side of \eqref{est1teomorrey}. By the definitions of the weighted Lebesgue norm and the weight \(\omega\), it follows that
\begin{eqnarray}
\|f\|_{L^{q}_{\omega}(\Omega)}&=&\left(\int_{\mathbb R^n} |f(x)|^q \omega(x)dx\right)^{\frac{1}{q}}\leq\left(\int_{B_r(x_0)} |f(x)|^q\omega(x) dx\right)^{\frac{1}{q}}+ \left(\int_{\left(B_r(x_0)\right)^c} |f(x)|^q \omega(x)dx\right)^{\frac{1}{q}}dt\nonumber\\
&=:&\mathfrak{J}_{1}+\mathfrak{J}_{2}.\label{estdef}
\end{eqnarray}
We estimate  the terms \(\mathfrak{J}_{1}\) and \(\mathfrak{J}_{2}\) seperately.
Since \(\omega_{|_{\mathrm{B}_{r}(x_{0})}}=r^{-n+\theta-\tau}\), it follows that
\begin{eqnarray}\label{estJ1}
\mathfrak{J}_{1}&=&\left(\int_{B_r(x_0)} |f(x)|^q \omega(x)dx\right)^{\frac{1}{q}}=\left(\int_{B_r(x_0)} |f(x)|^q\,r^{(-n+\bar\theta-\tau)} dx\right)^{\frac{1}{q}}\leq r^{\frac{-\tau}{q}} \|f\|_{L^{q,\bar\theta}(\Omega)}.
\end{eqnarray}
Regarding $\mathfrak{J}_{2}$, we need to carry out a delicate analysis.  By the definition of the weight \( \omega \), we know that \( \omega \leq r^{-n+\bar\theta-\tau} \) in \( \mathbb{R}^{n} \), and that \( \omega(x) = |x - x_{0}|^{-n+\bar\theta-\tau} \) for all \( x \in (\mathrm{B}_{r}(x_{0}))^{c} \). Thus, using that \(\frac{1}{-n+\theta-\tau}<0\),  we infer 
\[
\omega(x)>s\Longrightarrow s^{\frac{1}{-n+\bar\theta-\tau}}> r, \quad \mbox{for all } s > 0, \  x \in (\mathrm{B}_{r}(x_{0}))^{c}.
\]
Consequently, it yields  that
\begin{eqnarray*}
|x-x_{0}|<s^{\frac{1}{-n+\theta-\tau}},   \quad \mbox{for all }  x\in (\mathrm{B}_{r}(x_{0}))^{c}.
\end{eqnarray*}
Hence,  for all \( s > 0 \), we obtain
\begin{eqnarray}\label{inclusaodebola}
\{x\in (\mathrm{B}_{r}(x_{0}))^{c};\omega(x)>s\}\subset \mathrm{B}_{s^{\frac{1}{-n+\bar\theta-\tau}}}(x_{0}).
\end{eqnarray}
 In addition, for each \(j\in\mathbb{N}\), we define the annulus 
\(
\mathrm{A}_{j}(x_{0},r)=\mathrm{B}_{2^{j}r}(x_{0})\setminus \mathrm{B}_{2^{j-1}r}(x_{0}).
\) Thus,
\begin{eqnarray}
\mathfrak{J}_{2}&=&\left(\sum_{j=1}^{\infty}\int_{\mathrm{A}_{j}(x_{0},r)}|f(x)|^q \omega(x)dx\right)^{\frac{1}{q}}.
%\nonumber\\
%&\leq&\sum_{j=1}^{\infty}q\int_{0}^{\infty}t^{q-1}\left(\int_{\mathrm{A}_{j}(x_{0},r)}\omega(x)\chi_{\{|f|>t\}}(x)dx\right)^{\frac{q}{p}}dt,
\label{estdeJ2parapmaiorqueq}
\end{eqnarray}
%where in last inequality we use the convexity inequality for series since \(\frac{q}{p}\leq 1\).
We now analyze each one of the integrals into the series in \eqref{estdeJ2parapmaiorqueq}. Since  \(n-\bar\theta> 0\), for each \(j\in\mathbb{N}\), we have 
\[
\omega(x)=|x-x_{0}|^{-n+\bar\theta-\tau}\leq 2^{(j-1)(-n+\bar\theta-\tau)}r^{-n+\bar\theta-\tau}=2^{-(j-1)(n-\bar\theta+\tau)}r^{-n+\bar\theta-\tau}, \mbox{ for all }  x\in \mathrm{A}_{j}(x_{0},r).
\]
 So, using this along with the fact that $f\equiv 0$ in $\mathbb R^n\setminus \Omega$, and the definition of the {\it Morrey norm}, we deduce
\begin{eqnarray}
(2^{j-1}r)^{(n-\bar\theta+\tau)}\int_{\mathrm{A}_{j}(x_{0},r)}|f(x)|^q \omega(x)dx&\leq&\int_{\mathrm{A}_{j}(x_{0},r)}|f(x)|^qdx\leq\int_{\mathrm{B}_{2^{j}r}(x_{0})}|f(x)|^qdx\nonumber\\
&=&\int_{\Omega(x_{0},2^{j}r)} |f(x)|^q dx\leq(2^jr)^{n-\bar\theta}\|f\|_{L^{q,\bar\theta}(\Omega)}^q.
%\nonumber\\&=&2^{n}2^{-(j-1)\tau-\bar\theta}r^{-\tau}\|f\|_{L^{q,\bar\theta}(\Omega)}^q .
\label{estdeAj}
\end{eqnarray}
Plugging \eqref{estdeAj} into \eqref{estdeJ2parapmaiorqueq}, it yields
\begin{eqnarray}
\mathfrak{J}_{2}&\leq&\sum_{j=1}^{\infty}2^{\frac{-(j-1)\tau+n -\bar\theta}{q}}r^{\frac{-\tau}{q}}\|f\|_{L^{q,\bar\theta}(\Omega)}=2^{\frac{(n+\tau-\bar\theta)}{q}}r^{\frac{-\tau}{q}}\|f\|_{L^{q,\bar\theta}(\Omega)}\sum_{j=1}^{\infty}\frac{1}{\left(2^{\frac{\tau}{q}}\right)^{j}}\nonumber\\
&=&\frac{2^{\frac{(n+\tau-\bar\theta)}{q}}}{2^{\frac{\tau}{q}}-1}r^{\frac{-\tau}{q}}\|f\|_{L^{q,\bar\theta}(\Omega)}\label{estdeJ2}.
\end{eqnarray}
Thus, from \eqref{estJ1} and \eqref{estdeJ2}, it yields that
\begin{align}\label{f-est}
\|f\|_{L^{q}_{\omega}(\Omega)}\leq 2\max\left\{1, \frac{2^{\frac{(n+\tau-\bar\theta)}{q}}}{2^{\frac{\tau}{q}}-1}\right\}r^{\frac{-\tau}{q}}\|f\|_{L^{q,\bar\theta}(\Omega)}.
\end{align}

Next step is estimating the second expression on the right-hand side of \eqref{est1teomorrey}. By the definitions of the weighted Lebesgue norm and the weight \(\omega\), it follows that
\begin{eqnarray}
\|\varphi_u\|_{L^{q}_{\omega}(\Omega)}&=&\left(\int_{\mathbb R^n} |\varphi(|\{z\in \Omega\,:\,u(z)\geq u(x)\}|)|^q \omega(x)dx\right)^{\frac{1}{q}}\nonumber\\
&\leq&\left(\int_{B_r(x_0)} |\varphi(|\{z\in \Omega\,:\,u(z)\geq u(x)\}|)|^q\omega(x) dx\right)^{\frac{1}{q}}\nonumber\\
&+& \left(\int_{\left(B_r(x_0)\right)^c} |\varphi(|\{z\in \Omega\,:\,u(z)\geq u(x)\}|)|^q \omega(x)dx\right)^{\frac{1}{q}}\label{estdeff}
\end{eqnarray}
Now, following a similar calculation as in \eqref{estdef}, \eqref{estJ1} and \eqref{estdeJ2}, from \eqref{estdeff} we achieve that 
\begin{align}\label{varphi-est}
  \|\varphi_u\|_{L^{q}_{\omega}(\Omega)}\leq 2\max\left\{1, \frac{2^{\frac{(n+\tau-\bar\theta)}{q}}}{2^{\frac{\tau}{q}}-1}\right\}r^{\frac{-\tau}{q}}\|\varphi_u\|_{L^{q,\bar\theta}(\Omega)}.  
\end{align}
Consequently, combining \eqref{f-est} and \eqref{varphi-est}, in view of \eqref{est1teomorrey}, we obtain
\begin{eqnarray}\label{est2teomorrey}
\|D^{2}u\|_{L^{q}(\Omega(x_{0},r))}\leq 2\mathrm{C}r^{\frac{n-\bar\theta}{q}}\max\left\{1, \frac{2^{\frac{(n+\tau-\bar\theta)}{q}}}{2^{\frac{\tau}{q}}-1}\right\}\left(\|f\|_{L^{q,\bar\theta}(\Omega)}+\|\varphi_u\|_{L^{q,\bar\theta}(\Omega)}\right).
\end{eqnarray}
By the arbitrariness of \(\tau \in (0, \theta)\), taking \(\tau \to \theta^{-}\), we obtain from \eqref{est2teomorrey} that
\begin{eqnarray}\label{est3teomorrey}
\|D^{2}u\|_{L^{q}(\Omega(x_{0},r))}\leq \mathrm{C}^{\prime}r^{\frac{n-\bar\theta}{q}}\left(\|f\|_{L^{q,\bar\theta}(\Omega)}+\|\varphi_u\|_{L^{q,\bar\theta}(\Omega)}\right),
\end{eqnarray}
where
\[
\mathrm{C}^{\prime}=2\mathrm{C}\max\left\{1,\frac{2^{\frac{n}{q}}}{2^{\frac{\bar\theta}{q}}-1}\right\}.
\]
Finally, dividing both sides of \eqref{est3teomorrey} by \( r^{\frac{n-\bar\theta}{q}} \) and taking the supremum over \( x_{0} \in \Omega \) and \( 0 < r \leq \operatorname{diam}(\Omega) \), we  conclude
\[
\|D^{2}u\|_{L^{q,\bar\theta}(\Omega)} \leq \mathrm{C}^{\prime}\left(\|f\|_{L^{q,\bar\theta}(\Omega)}+\|\varphi_u\|_{L^{q,\bar\theta}(\Omega)}\right).
\]  
 The \( L^{q,\bar\theta}(\Omega) \)-regularity   of the terms \( u \) and \( Du \) follows in a completely analogous way to the argument established for the Hessian, by applying the weighted estimates provided via Theorem \ref{T1}.    
\end{proof}

As an immediate consequence of Theorem \ref{T3}, we recover the Hölder regularity for the gradient of the solutions to problem \eqref{1.1}. While the case $\bar\theta=0$ follows from Corollary 3.4, the next result treats the genuinely Morrey regime $0<\bar\theta<n$.
\begin{corollary}
Under the assumptions of Theorem \ref{T3}, let \(u\) be a $L^{q}$-viscosity solution to problem \eqref{1.1} with $\gamma=g=0$. Then, $Du\in C^{0,1-\frac{n-\bar\theta}{q}}(\overline{\Omega})$.
\end{corollary}

%%%%%%%%%%%%%%%%%%%%%%%%%%%%%%%%%%%%%%%%%%%%%%%%%%%%%%%%%%%%%%
\section{\texorpdfstring{Global $L^{\Upsilon}_{\omega}-\mathrm{BMO}$ Estimates}{Global L^{Upsilon}_{omega}-BMO estimates}}\label{Section5}

In this section, we present the proof of Theorem \ref{T2}. Similarly to the approach used for Theorem \ref{T1}, it is crucial to guarantee the global regularity in the absence of the semilinear term $\varphi_{u}$. This result is summarized in the following theorem.

\begin{theorem}\label{Casevarphiequaltozero}
Assume the structural conditions {\bf$(A_1)$-($A_3$)} and {\bf ($A_4$)$^{\star}$} hold true. Let $u$ be a $L^{p}$-viscosity solution to 
 \begin{equation*}
\left\{
\begin{array}{rclcl}
 F(D^2u,Du,u,x) &=& f(x)& \mbox{in} &   \Omega \\
\mathcal{G}(Du,u,x)&=& g(x) &\mbox{on}& \partial \Omega,
\end{array}
\right.
\end{equation*}
where $p=\bar{p}n$, and  $\bar{p}>1$ is the constant of  Lemma \ref{mergulhoorliczlebesgue}, and $f\in L^{\Upsilon}_{\omega}-\mathrm{BMO}(\Omega)\cap L^{\Upsilon}_{\omega}(\Omega)$ for $\Upsilon(t)=\Phi(t^{n})$. Then, $D^{2}u\in L^{\Upsilon}_{\omega}-\mathrm{BMO}(\Omega)$ with the following estimate
\begin{equation*}
\|D^{2}u\|_{L^{\Upsilon}_{\omega}-\mathrm{BMO}(\Omega)}\leq \mathrm{C}(\|u\|_{L^{\infty}(\Omega)}^{n}+\|f\|_{L^{\Upsilon}_{\omega}-\mathrm{BMO}(\Omega)}+\|g\|_{C^{1,\alpha}(\partial\Omega)}),
\end{equation*}
for a positive constant $\mathrm{C}$ depending only on $n$, $\lambda$, $\Lambda$, $a$, $b$, $\bar{p}$, $p_{2}$, $\Phi$, $\omega$, $\delta_{0}$, $\bar{\theta}$, $\alpha^{\prime}$, $\|\beta\|_{C^{2}(\partial \Omega)}$, $\|\gamma\|_{C^{2}(\partial \Omega)}$ and $\|\partial \Omega\|_{C^{2,\alpha}}$.
\end{theorem}
\begin{proof}
Based on standard arguments (cf. \cite{Winter} and \cite{Bessa}), the result follows from reductions to simplify the problem and from the use of boundary estimates for the Hessian, already available in \cite[Theorem 5.2]{Bessa}. To avoid unnecessary repetition, we only briefly comment on the details of this technique. First, note that $u$ is also an $L^{p}$-viscosity solution to
\begin{equation*}
\left\{
\begin{array}{rclcl}
G(D^2u,x) &=& \widehat{f}(x)& \mbox{in} &   \Omega \\
\mathcal{G}(Du,u,x)&=& g(x) &\mbox{on}& \partial \Omega,
\end{array}
\right.
\end{equation*}
where $G(\mathrm{M},x)=F(\mathrm{M},0,0,x)$ and $\widehat{f}(x)=:F(D^{2}u(x),0,0,x)$ is well-defined, since by \cite[Theorem 3.6]{CCKS} we know that $u$ is twice differentiable a.e. in $\Omega$. Moreover, by the structural hypothesis {\bf$(A_1)$}, we have that $\widehat{f}$ satisfies
\begin{eqnarray}
|\widehat{f}(x)|&\leq& |F(D^{2}u(x),Du(x),u(x),x)-F(D^{2}u(x),0,0,x)|+|f(x)|\nonumber\\
&\leq& a|Du(x)|+b|u(x)|+|f(x)|,\,\,\text{ a. e. }\quad x\in \Omega.\label{5.1}
\end{eqnarray}
On the other hand, by the weighted Orlicz--Poincar\'e inequality \cite{Yong25}, we obtain the following estimates
\begin{equation}\label{5.2}
\|u\|_{L^{\Upsilon}_{\omega}-\mathrm{BMO}(\Omega)}\leq \mathrm{c}_{1} \|Du\|_{L^{\Upsilon}_{\omega}(\Omega)}\quad \text{and}\quad \|Du\|_{L^{\Upsilon}_{\omega}-\mathrm{BMO}(\Omega)}\leq \mathrm{c}_{2}\|D^{2}u\|_{L^{\Upsilon}_{\omega}(\Omega)},
\end{equation}
where $\mathrm{c}_{1},\mathrm{c}_{2}>0$ depends only on $n$, $\omega$ and $\Phi$. Since $f\in L^{\Upsilon}_{\omega}(\Omega)$,  by \cite[Theorem 3.5]{Bessa}, we know that these quantities are bounded, specifically,
\begin{equation*}
\|Du\|_{L^{\Upsilon}_{\omega}(\Omega)}, \ \|D^2u\|_{L^{\Upsilon}_{\omega}(\Omega)}\leq \mathrm{C}(\|f\|_{L^{\Upsilon}_{\omega}(\Omega)}+\|g\|_{C^{1,\alpha}(\partial \Omega)}),
\end{equation*}
where $\mathrm{C}>0$ depends only on \(n\), \(\lambda\), \(\Lambda\), \(a\), \(b\), \(\delta_{0}\), \(\bar{p}\), \(p_{2}\), \(\Phi\), \(\omega\), \(\bar{\theta}\), \(\alpha^{\prime}\), \(\|\beta\|_{C^{2}(\partial\Omega)}\), \(\|\gamma\|_{C^{2}(\partial \Omega)}\), \(\partial \Omega \) and \(\operatorname{diam}(\Omega)\). Combining this fact with the estimates \eqref{5.1} and \eqref{5.2}, we may conclude that
\begin{equation*}
\|\widehat{f}\|_{L^{\Upsilon}_{\omega}-\mathrm{BMO}(\Omega)}\leq \mathrm{C}(\|f\|_{L^{\Upsilon}_{\omega}-\mathrm{BMO}(\Omega)}+\|g\|_{C^{1,\alpha}(\partial \Omega)}).
\end{equation*}
Thus, we remove the dependence of the lower-order terms ($u$ and $Du$) from the governing operator, and by a standard boundary flattening argument it suffices to consider  the case where $u$ is a solution to
\begin{equation*}
\left\{
\begin{array}{rclcl}
G(D^2u,x) &=& \widehat{f}(x)& \mbox{in} &   \mathrm{B}^{+}_{1} \\
\mathcal{G}(Du,u,x)&=& g(x) &\mbox{on}& \mathrm{T}_{1}.
\end{array}
\right.
\end{equation*}
In this case, since $G$ also satisfies the structural hypothesis {\bf ($A_4$)$^{\star}$}, we may invoke the $L^{\Upsilon}_{\omega}$-$\mathrm{BMO}$ estimates for the problem found in \cite[Theorem 5.2]{Bessa} and, by a standard covering theorem, we conclude that $D^{2}u\in L^{\Upsilon}_{\omega}$-$\mathrm{BMO}(\Omega)$ and satisfies the following estimate
\begin{equation*}
\|D^{2}u\|_{L^{\Upsilon}_{\omega}-\mathrm{BMO}(\Omega)}\leq \mathrm{C}(\|u\|_{L^{\infty}(\Omega)}^{n}+\|f\|_{L^{\Upsilon}_{\omega}-\mathrm{BMO}(\Omega)}+\|g\|_{C^{1,\alpha}(\partial\Omega)}),
\end{equation*}
for $\mathrm{C}>0$ depending only on $n$, $\lambda$, $\Lambda$, $a$, $b$, $\bar{p}$, $p_{2}$, $\Phi$, $\omega$, $\delta_{0}$, $\mathrm{C}_{1}$, $\mathrm{C}_{2}$, $\bar{\theta}$, $\alpha^{\prime}$, $\|\beta\|_{C^{2}(\partial \Omega)}$, $\|\gamma\|_{C^{2}(\partial \Omega)}$ and $\|\partial \Omega\|_{C^{2,\alpha}}$, where $\mathrm{C}_{1}$ and $\mathrm{C}_{2}$ are the constants given by Remark \ref{Remark2.12}.
\end{proof}

\begin{remark}
By the assumptions required for the weighted Orlicz--$\mathrm{BMO}$ estimates in \cite{Bessa}, Theorem \ref{Casevarphiequaltozero} remains valid in case $\beta$ and $\gamma$ are merely of class $C^{1,\alpha}$, instead of $C^{2}$, and the dependence on the constants is modified accordingly.
\end{remark}

As a consequence of this result, we obtain a refinement of the estimates.
\begin{corollary}\label{cor:hes-BMO}
Assume that the structural conditions {\bf$(A_1)$-($A_3$)} and {\bf ($A_4$)$^{\star}$} hold. Let $u$ be a $L^{p}$-viscosity solution to 
 \begin{equation*}
\left\{
\begin{array}{rclcl}
 F(D^2u,Du,u,x) &=& f(x)& \mbox{in} &   \Omega \\
\mathcal{G}(Du,u,x)&=& g(x) &\mbox{on}& \partial \Omega,
\end{array}
\right.
\end{equation*}
where $p=\bar{p}n$, and  $\bar{p}>1$ is the constant given by Lemma \ref{mergulhoorliczlebesgue}, and $f\in L^{\Upsilon}_{\omega}-\mathrm{BMO}(\Omega)\cap L^{\Upsilon}_{\omega}(\Omega)$ for $\Upsilon(t)=\Phi(t^{n})$. Then, $D^{2}u\in L^{\Upsilon}_{\omega}-\mathrm{BMO}(\Omega)$ and satisfies the following estimate
\begin{equation}\label{5.3}
\|D^{2}u\|_{L^{\Upsilon}_{\omega}-\mathrm{BMO}(\Omega)}\leq \mathrm{C}(\|f\|_{L^{\Upsilon}_{\omega}-\mathrm{BMO}(\Omega)}+\|g\|_{C^{1,\alpha}(\partial\Omega)}),
\end{equation}
for a positive constant $\mathrm{C}$ depending only on $n$, $\lambda$, $\Lambda$, $a$, $b$, $\bar{p}$, $p_{2}$, $\Phi$, $\omega$, $\delta_{0}$, $\bar{\theta}$, $\alpha^{\prime}$, $\|\beta\|_{C^{2}(\partial \Omega)}$, $\|\gamma\|_{C^{2}(\partial \Omega)}$ and $\|\partial \Omega\|_{C^{2,\alpha}}$.
\end{corollary}
\begin{proof}
By Theorem \ref{Casevarphiequaltozero}, we know that $D^{2}u\in L^{\Upsilon}_{\omega}-\mathrm{BMO}(\Omega)$, so it remains to show that the estimate \eqref{5.3} holds. We argue by  contradiction, if \eqref{5.3} does not hold, there exist sequences $(F_{j})_{j\in\mathbb{N}}$, $(u_{j})_{j\in\mathbb{N}}$, $(f_{j})_{j\in\mathbb{N}}$, and $(g_{j})_{j\in\mathbb{N}}$ satifying
\begin{equation*}
\left\{
\begin{array}{rclcl}
 F_{j}(D^2u_{j},Du_{j},u_{j},x) &=& f_{j}(x)& \mbox{in} &   \Omega \\
\mathcal{G}(Du_{j},u_{j},x)&=& g_{j}(x) &\mbox{on}& \partial \Omega,
\end{array}
\right.
\end{equation*}
where each $F_{j}$ satisfies the assumptions {\bf $(A_1)$-($A_3$)} and {\bf ($A_4$)$^{\star}$} with the same constants, $f_{j}\in L^{\Upsilon}_{\omega}-\mathrm{BMO}(\Omega)\cap L^{\Upsilon}_{\omega}(\Omega)$ and $g_{j}\in C^{1,\alpha}(\partial \Omega)$. However,
\begin{equation}\label{5.4}
\|D^{2}u_{j}\|_{L^{\Upsilon}_{\omega}-\mathrm{BMO}(\Omega)}> j(\|f_{j}\|_{L^{\Upsilon}_{\omega}-\mathrm{BMO}(\Omega)}+\|g_{j}\|_{C^{1,\alpha}(\partial\Omega)})
\end{equation}
for any $j\in\mathbb{N}$. By a normalization argument, we may assume that $\|D^{2}u_{j}\|_{L^{\Upsilon}_{\omega}-\mathrm{BMO}(\Omega)}= 1$ for all $j$. In this case, from \eqref{5.4} we have that $f_{j}\to 0$ in $L^{\Upsilon}_{\omega}-\mathrm{BMO}(\Omega)$ and $g_{j}\to 0$ in $C^{1,\alpha}(\partial \Omega)$. The convergence of the source terms, combined with the embedding of the $p$-$\mathrm{BMO}$ spaces (cf. \cite[Chapter IV, \S 1.1.3]{Stein}) and Remark \ref{Remark2.12}, imply that $f_{j}\to 0$ in $p$-$\mathrm{BMO}$, for $p=\bar{p}n$ as in the statement. Consequently, since $f_{j}\in L^{\Upsilon}_{\omega}(\Omega)\subset L^{p}(\Omega)$ (by Lemma \ref{mergulhoorliczlebesgue} and the definition of $\Upsilon$), it follows from \cite[Chapter IV, \S 2]{Stein} that
\begin{equation*}
\|f_{j}\|_{L^{p}(\Omega)}\leq \mathrm{C}_{p}\|f_{j}\|_{p-\mathrm{BMO}(\Omega)}\to 0.
\end{equation*}
From this convergence and the boundary data, we may invoke the global $W^{2,p}$ estimates for each $u_{j}$ given by \cite[Theorem 3.5]{Bessa}, in the case $\Phi(t)=t^{\bar{p}}$ and $\omega\equiv 1$, and guarantee that there exists a constant $\mathrm{C}>0$, independent of $j$, such that
\begin{equation}\label{5.5}
\|u_{j}\|_{W^{2,p}(\Omega)}\leq \mathrm{C}\big(\|f_{j}\|_{L^{p}(\Omega)}+\|g_{j}\|_{C^{1,\alpha}(\partial \Omega)}\big)
\end{equation}
for all $j\in\mathbb{N}$. Hence, up to a subsequence, there exists a function $u_{\infty}$ such that $u_{j}\to u_{\infty}$ in $W^{2,p}(\Omega)$ as $j\to +\infty$.
Further, by the structural properties, $F_{j}\to F_{\infty}$, locally uniformly as $j\to +\infty$, where $F_{\infty}$ is a $(\lambda,\Lambda,a,b)$-elliptic operator. Consequently, by the Sobolev embeddings and the convergences above, we may conclude by the  stability theory that $u_{\infty}$ solves
\begin{equation*}
\left\{
\begin{array}{rclcl}
 F_{\infty}(D^2u_{\infty},Du_{\infty},u_{\infty},x) &=& 0& \mbox{in} &   \Omega \\
\mathcal{G}(Du_{\infty},u_{\infty},x)&=& 0 &\mbox{on}& \partial \Omega.
\end{array}
\right.
\end{equation*}
Since the operator $\mathcal{G}=\mathcal{G}(\vec{\xi},s,x)$ is of class $C^{2}$ with respect to $\vec{\xi}$ and $x$, we may conclude from \cite[Theorem 7.19]{Lieberman} that $u_{\infty}\equiv 0$. Finally, using the estimate in Theorem \ref{Casevarphiequaltozero} and the convergences above, it follows that
\begin{eqnarray*}
1=\|D^{2}u_{\infty}\|_{L^{\Upsilon}_{\omega}-\mathrm{BMO}(\Omega)}\leq \mathrm{C}\liminf_{j\to \infty}(\|u_{j}\|_{L^{\infty}(\Omega)}^{n}+\|f_{j}\|_{L^{\Upsilon}_{\omega}-\mathrm{BMO}(\Omega)}+\|g_{j}\|_{C^{1,\alpha}(\partial\Omega)})=0,  
\end{eqnarray*}
which is a contradiction. This completes the proof.
\end{proof}
We are now in a position to prove Theorem \ref{T2}.
\begin{proof}[\bf Proof of the Theorem \ref{T2}]
By applying  Theorem \ref{T1}, we already have the existence of a $L^p$-viscosity solution to the problem \eqref{1.1}. From Corollary \ref{cor:hes-BMO}, we can take a sequence of functions $\{u_\eps\}_{\eps>0}$ with $\{D^2 u_\eps\}_{\eps>0} \subset L^{\Upsilon}_{\omega}-\mathrm{BMO}(\Omega)$ satisfying the problem
\eqref{penprob}, with penalized source term. Thus, we are able to follow the same steps in the proof of Theorem \ref{T1}, obtaining at the end of this proof, 
\[
\|D^2 u_\eps\|_{L^{\Upsilon}_{\omega}-\mathrm{BMO}}
\le C \left(   \|f\|_{L^{\Upsilon}_{\omega}-\mathrm{BMO}} +\|\varphi_{u_{\varepsilon},\varepsilon}\|_{L^{\Upsilon}_{\omega}-\mathrm{BMO}(\Omega)}
+\|g\|_{C^{1,\alpha}(\partial\Omega)}\right),
\]
instead of \eqref{est1teo1.2},
 for a  constant $\mathrm{C}>0$ depending only on $n$, $\lambda$, $\Lambda$, $a$, $b$, $\bar{p}$, $p_{2}$, $\Phi$, $\omega$, $\delta_{0}$, $\bar{\theta}$, $\alpha^{\prime}$, $\|\beta\|_{C^{2}(\partial \Omega)}$, $\|\gamma\|_{C^{2}(\partial \Omega)}$ and $\|\partial \Omega\|_{C^{2,\alpha}}$. Passing to the limit, as $\eps \to 0$, we obtain
 the following estimate
\begin{equation*}
\|D^{2}u\|_{L^{\Upsilon}_{\omega}-\mathrm{BMO}(\Omega)}\leq \mathrm{C}(\|f\|_{L^{\Upsilon}_{\omega}-\mathrm{BMO}(\Omega)}+\|\varphi_{u}\|_{L^{\Upsilon}_{\omega}-\mathrm{BMO}(\Omega)}+\|g\|_{C^{1,\alpha}(\partial\Omega)}).
\end{equation*}
Therefore,  the Hessian of the solutions  for the Grad-Mercier type equation \eqref{1.1} satisfy  $L^{\Upsilon}_{\omega}-\mathrm{BMO}(\Omega)$ estimates.
    
\end{proof}

\begin{remark}
We emphasize that the approach developed in the results above is fully flexible and may plausibly be adapted to derive global estimates in weighted Lorentz spaces (cf. \cite{BR}).
\end{remark}

An interesting consequence of Theorem \ref{T2} is the derivation of local $C^{1,\text{Log-Lip}}$ estimates for the solutions to problem \eqref{1.1}.

\begin{corollary}\label{Cor5.5}
Assume the hypotheses of Theorem \ref{T2} hold, and let $u$ be a solution to problem \eqref{1.1}. Then,
$u\in C^{1,\text{Log-Lip}}_{loc}(\Omega)$. More precisely, given $\Omega^{\prime}\Subset \overline{\Omega}$, we have
\begin{equation*}
|Du(x)-Du(y)|\leq \mathrm{C}(\|f\|_{L^{\Upsilon}_{\omega}-\mathrm{BMO}(\Omega)}+\|\varphi_{u}\|_{L^{\Upsilon}_{\omega}-\mathrm{BMO}(\Omega)}+\|g\|_{C^{1,\alpha}(\partial \Omega)})|x-y||\log|x-y||,\,\, \forall x,y\in \Omega^{\prime},\, x\neq y,
\end{equation*}
where the constant $\mathrm{C}>0$ depends only on $n$, $\lambda$, $\Lambda$, $a$, $b$, $\bar{p}$, $p_{2}$, $\Phi$, $\omega$, $\delta_{0}$, $\bar{\theta}$, $\alpha^{\prime}$, $\|\beta\|_{C^{2}(\partial \Omega)}$, $\|\gamma\|_{C^{2}(\partial \Omega)}$, $\Omega^{\prime}$ and $\|\partial \Omega\|_{C^{2,\alpha}}$.
\end{corollary}
\begin{proof}
First, we already know that $u\in C^{1}(\Omega)$, due to the Sobolev  embedding
$W^{2,\Upsilon}_{\omega}(\Omega)\hookrightarrow W^{2,p}(\Omega)$, since $p=\bar{p}n\in (n,\infty)$.
In this case, it remains  to prove that $Du\in C^{0,\text{Log-Lip}}_{loc}(\Omega)$. 
Indeed, by the global $L^{\Upsilon}_{\omega}-\mathrm{BMO}$ estimates for the Hessian and Remark \ref{Remark2.12}, we have that $D^{2}u\in \mathrm{BMO}(\Omega)$. 
Therefore, by the local Log-Lipschitz estimates \cite[Lemma 1]{AzzBed15} we obtain the desired regularity with the following estimate
\begin{equation*}
|Du(x)-Du(y)|\leq \mathrm{C}\|D^{2}u\|_{\mathrm{BMO}(\Omega)}|x-y||\log|x-y||, \,\, \forall x,y\in \Omega^{\prime},\, x\neq y.
\end{equation*}
Using again Remark \ref{Remark2.12} to estimate the $\mathrm{BMO}$-norm of the Hessian by its corresponding $L^{\Upsilon}_{\omega}-\mathrm{BMO}$ norm, and applying the estimate obtained by Theorem \ref{T2}, we conclude the desired result.
\end{proof}

\begin{remark}
 Corollary \ref{Cor5.5} ensures that the  solutions to problem \eqref{1.1} belong to the class $C^{1,1^{-}}_{\mathrm{loc}}$, that is, they are locally of class $C^{1,\nu}$ for every $0<\nu<1$.
\end{remark}

\subsection*{Declarations}

\subsubsection*{Funding}

J.S. Bessa has been supported by FAPESP-Brazil under Grant No. 2023/18447-3.  M. Soares has been supported by CNPq-Brazil under the Grants No. 303.154/2025-0 and  No. 152.786/2025-2.

\subsubsection*{Conflict of interest}

On behalf of all authors, the corresponding author states that there is no conflict of interest

\subsubsection*{Data availability statement}

Data availability does not apply to this article as no new data were created or analyzed in this study.

\bigskip
\noindent\textsc{Junior da Silva Bessa}\\
Departamento de Matem\'{a}tica\\
Instituto de Mamatem\'{a}tica, Estat\'{i}stica e Computa\c{c}\~{a}o Cient\'{i}fica - IMECC\\
Universidade Estadual de Campinas - Unicamp\\
Campinas, SP, Brazil\\
\noindent{Email address:\texttt{jbessa@unicamp.br}}
\bigskip
	
\noindent\textsc{Reshmi Biswas}\\
Departamento de Matem\'{a}tica\\
Universidad del Bío-Bío\\
Concepción,  Chile. \\
\noindent{Email address: \texttt{rbiswas@ubiobio.cl}}
\bigskip

\noindent\textsc{Mayra Soares}\\
Department de Matemática\\
Universidade de Bras\'{i}lia - UnB \\
Bras\'{i}lia, DF, Brazil\\
and\\
Departamento de Matem\'{a}tica\\
Instituto de Mamatem\'{a}tica, Estat\'{i}stica e Computa\c{c}\~{a}o Cient\'{i}fica - IMECC\\
Universidade Estadual de Campinas - Unicamp\\
Campinas, SP, Brazil\\
\noindent{Email address: \texttt{mayra.soares@unb.br}}


\begin{thebibliography}{99}

\bibitem{AzzBed15} J. Azzam and J. Bedrossian, Bounded mean oscillation and the uniqueness of active scalar equations, Trans. Amer. Math. Soc. {\bf 367} (2015), no. 5, 3095-3118.

\bibitem{Bessa24} J. da~S.~Bessa, Weighted Orlicz regularity for fully nonlinear elliptic equations with oblique derivative at the boundary via asymptotic operators, J. Funct. Anal. {\bf 286} (2024), no.~4, Paper No. 110295, 35 pp.

\bibitem{Bessa} J. da~S.~Bessa, J. V. ~ da Silva, M. N. B. Frederico and G. C. Ricarte, Sharp Hessian estimates for fully nonlinear elliptic equations under relaxed convexity assumptions, oblique boundary conditions and applications, J. Differential Equations {\bf 367} (2023), 451--493.

\bibitem{BR}  J. da~S.~Bessa and G.~C. Ricarte, Global weighted Lorentz estimates of oblique tangential derivative problems for weakly convex fully nonlinear operators, Potential Anal. {\bf 62} (2025), no.~4, 817--842.

\bibitem{BRS} J. da~S.~Bessa, G.~C. Ricarte and P.~H.~C. Silva, Optimal gradient regularity to degenerate fully nonlinear elliptic models with oblique boundary condition, Nonlinear Anal. {\bf 262} (2026), Paper No. 113919, 16 pp.

\bibitem{Brezis} H.~R. Brezis, {\it Functional analysis, Sobolev spaces and partial differential equations}, Universitext, Springer, New York, 2011.

\bibitem{BH20} S.-S. Byun and J. Han, $W^{2,p}$-estimates for fully nonlinear elliptic equations with oblique boundary conditions, J. Differential Equations {\bf 268} (2020), no.~5, 2125--2150.

\bibitem{BJ1} S.-S. Byun, J. Han and J. Oh, On $W^{2,p}$-estimates for solutions of obstacle problems for fully nonlinear elliptic equations with oblique boundary conditions, Calc. Var. Partial Differential Equations {\bf 61} (2022), no.~5, Paper No. 162, 15 pp.

\bibitem{BLOK} S.-S. Byun, M. Lee and J. Ok, Weighted regularity estimates in Orlicz spaces for fully nonlinear elliptic equations, Nonlinear Anal. {\bf 162} (2017), 178--196.

\bibitem{BKO} S.-S. Byun, H. Kim and J. Oh, $C^{1,\alpha }$ regularity for degenerate fully nonlinear elliptic equations with oblique boundary conditions on $C^1$ domains, Calc. Var. Partial Differential Equations {\bf 64} (2025), no.~5, Paper No. 174, 20 pp.

\bibitem{CC} L.~\'A. Caffarelli and X. Cabr\'e, {\it Fully nonlinear elliptic equations}, American Mathematical Society Colloquium Publications, 43, Amer. Math. Soc., Providence, RI, 1995.

\bibitem{CCKS} L.~\'A. Caffarelli, M.G. Crandall, M. Kocan, A. \'{S}wi\c{e}ch, On viscosity solutions of fully nonlinear equations with measurable ingredients, Comm. Pure Appl. Math. {\bf 49} (1996), no.~4, 365--397.

\bibitem{CafFarRes} L.~\'A. Caffarelli, A. Farah and D. Restrepo, On the Grad-Mercier equation and Semilinear Free Boundary Problems.  Preprint, arXiv:2504.12548v1 [math.AP] (2025). 

\bibitem{CaffToma21} L.~\'A. Caffarelli and I. Tomasetti, Fully nonlinear equations with applications to Grad equations in plasma physics, Comm. Pure Appl. Math. {\bf 76} (2023), no.~3, 604--615.

\bibitem{DHHR11} L.~Diening, P.~Harjulehto, P~ Hästö and M.~ Růžička,  {\it Lebesgue and Sobolev spaces with variable exponents}, Lecture Notes in Mathematics, 2017, Springer, Heidelberg, 2011

\bibitem{fioreza} A. Fiorenza and M. Krbec, Indices of Orlicz spaces and some applications, Comment. Math. Univ. Carolin. {\bf 38} (1997), no.~3, 433--451.

\bibitem{Grad79} H. Grad,  \textit{Alternating dimension plasma transport in three dimensions}. Computing methods in applied sciences and engineering (Proc. Fourth Internat. Sympos., Versailles, 1979), 261–284. North-Holland, Amsterdam–New York, 1980.

\bibitem{GS01} Y. Giga and M.-H. Sato, On semicontinuous solutions for general Hamilton-Jacobi equations, Proc. Japan Acad. Ser. A Math. Sci. {\bf 75} (1999), no.~9, 159--162.

%\bibitem{Ho}K.-P. Ho, Atomic decomposition of Hardy spaces and characterization of BMO via Banach function spaces, Anal. Math. {\bf 38} (2012), no.~3, 173--185.

\bibitem{KokiMiro} V.~M. Kokilashvili and M. Krbec, {\it Weighted inequalities in Lorentz and Orlicz spaces}, World Sci. Publ., River Edge, NJ, 1991.

\bibitem{LaurenceStredulinsky} P. Laurence and E.~W. Stredulinsky, A new approach to queer differential equations, Comm. Pure Appl. Math. {\bf 38} (1985), no.~3, 333--355.

\bibitem{LLO} A.~K. Lerner, E. Lorist and S.~J. Ombrosi, BMO with respect to Banach function spaces, Math. Ann. {\bf 388} (2024), no.~4, 4053--4082.

\bibitem{LiZhang}D.~S. Li and K. Zhang, Regularity for fully nonlinear elliptic equations with oblique boundary conditions, Arch. Ration. Mech. Anal. {\bf 228} (2018), no.~3, 923--967.

\bibitem{Lieberman} G.~M. Lieberman, {\it Oblique derivative problems for elliptic equations}, World Sci. Publ., Hackensack, NJ, 2013.

\bibitem{Lieb01} G.~M. Lieberman, On the H\"older gradient estimate for solutions of nonlinear elliptic and parabolic oblique boundary value problems, Comm. Partial Differential Equations {\bf 15} (1990), no.~4, 515--523.


\bibitem{MossinoTemam} J. Mossino and R.~M. Temam, Directional derivative of the increasing rearrangement mapping and application to a queer differential equation in plasma physics, Duke Math. J. {\bf 48} (1981), no.~3, 475--495.


\bibitem{Muckenhoupt}B. Muckenhoupt, Weighted norm inequalities for the Hardy maximal function, Trans. Amer. Math. Soc. {\bf 165} (1972), 207--226.

\bibitem{pimentel} E.~A. Pimentel and E.~V.~O. Teixeira, Sharp Hessian integrability estimates for nonlinear elliptic equations: an asymptotic approach, J. Math. Pures Appl. (9) {\bf 106} (2016), no.~4, 744--767.

\bibitem{Stein} E.~M. Stein, {\it Harmonic analysis: real-variable methods, orthogonality, and oscillatory integrals}, Princeton Mathematical Series Monographs in Harmonic Analysis, 43 III, Princeton Univ. Press, Princeton, NJ, 1993.

\bibitem{SteinShak05}E.~M. Stein and R. Shakarchi, {\it Real analysis}, Princeton Lectures in Analysis, 3, Princeton Univ. Press, Princeton, NJ, 2005.

\bibitem{Temam} R.~M. Temam, Monotone rearrangement of a function and the Grad-Mercier equation of plasma physics, in {\it Proceedings of the International Meeting on Recent Methods in Nonlinear Analysis (Rome, 1978)}, pp. 83--98, Pitagora, Bologna, 1979.

\bibitem{Winter} N. Winter, $W^{2,p}$ and $W^{1,p}$-estimates at the boundary for solutions of fully nonlinear, uniformly elliptic equations, Z. Anal. Anwend. {\bf 28} (2009), no.~2, 129--164.

\bibitem{Yong25} L. Yong, Weighted Orlicz-Poincar\'e{} inequalities in product spaces, J. Math. Anal. Appl. {\bf 551} (2025), no.~2, Paper No. 129680, 9 pp..

\bibitem{ZJMZ} Y.~Zhang, X.~ Jin, L.~ Ma and Z.~Zhang,  Global Calder\'on-Zygmund estimates for asymptotically convex fully nonlinear Grad-Mercier type equations, Nonlinearity {\bf 39} (2026), no.~2, Paper No. 025011, 19 pp.

\bibitem{ZZ22}J. Zhang and S. Zheng, Weighted Lorentz estimates for fully nonlinear elliptic equations with oblique boundary data, J. Elliptic Parabol. Equ. {\bf 8} (2022), no.~1, 255--281.

\bibitem{ZZZ21} J. Zhang, S. Zheng and C. Zuo, $W^{2,p}$-regularity for asymptotically regular fully nonlinear elliptic and parabolic equations with oblique boundary values, Discrete Contin. Dyn. Syst. Ser. S {\bf 14} (2021), no.~9, 3305--3318.
\end{thebibliography}
\end{document}